\documentclass[11pt]{amsart}
\usepackage{amsmath}
\usepackage{amsfonts}
\usepackage{comment}
\usepackage{amsthm}
\usepackage{mathtools}
\usepackage{euscript}
\usepackage{dirtytalk}
\usepackage{amssymb}
\usepackage{graphicx}
\usepackage{mathrsfs}
\usepackage{enumitem}
\newcommand{\la}{\lambda}

\newcommand\norm[1]{\left\lVert#1\right\rVert}

\newcommand{\lpnorm}[3]{\norm{#1}_{L^{#2}(#3)}}
\newcommand{\lpnorms}[3]{\Bigl|\Bigl|{#1}\Bigr|\Bigr|_{L^{#2}(#3)}}

\newcommand{\lpnb}[2]{\norm{#1}_{#2}}
\newcommand{\lpn}[2]{||{#1}||_{#2}}
\newcommand{\supp}[1]{\text{supp}\, #1}
\newcommand{\R}[1][]{\mathbb{R}^#1}
\newcommand{\T}[1][]{\mathbb{T}^#1}
\newcommand{\Z}[1][]{\mathbb{Z}^#1}

\newtheorem{theorem}{Theorem}[section]
\newtheorem{proposition}[theorem]{Proposition}
\newtheorem{conjecture}[theorem]{Conjecture}
\newtheorem{remark}[theorem]{Remark}

\newtheorem{lemma}[theorem]{Lemma}

\begin{document}

\title{Sharp Spectral Projection Estimates for the Torus at $p=\frac{2(n+1)}{n-1}$}
\author{Daniel Pezzi}

\begin{abstract}
    We prove sharp spectral projection estimates for general tori in all dimensions at the exponent $p_c=\frac{2(n+1)}{n-1}$ for shrinking windows of width $1$ down to windows of length $\la^{-1+\kappa}$ for fixed $\kappa>0$. This improves and  generalizes the work of Blair--Huang--Sogge which proved sharp results for windows of width $\la^{-\frac{1}{n+3}}$ in \cite{BHS1}, and the work of Hickman \cite{H1}, Germain--Myerson \cite{GM1}, and Demeter--Germain \cite{DG1} which proved results for windows of all widths but incurred a sub-polynomial loss. Our work uses the approaches of these two groups of authors, combining the bilinear decomposition and microlocal techniques of Blair--Huang--Sogge with the decoupling theory and explicit lattice point lemmas used by Hickman, Germain--Myerson, and Demeter--Germain to remove these losses.
\end{abstract}
\maketitle

\section{Introduction}
\subsection{Stating the problem}
The goal of this paper is to improve on spectral projection estimates for general tori. Let $\mathbb{T}^n = \R{n}/(e_1^*\mathbb{Z}+...+e_n^*\mathbb{Z})$ where $\{e_i^*\}$ is a basis of $\R{n}$. We will denote the corresponding lattice $L^*$, and all implicit constants will be allowed to depend on this lattice. For ease of computation we will also assume a scaling such that the injectivity radius is larger than $\frac{1}{2}$. Let $k_L = \sum_{i=1}^nk_ie_i$ for $k\in \Z{n}$ where $e_i$ is the dual basis vector to $e_i^*$. Let $L$ denote the dual lattice to $L^*$. An orthonormal basis of $L^2(\T{n})$ is given by

\begin{equation*}
    \{e_{k_L}(x)\coloneq c_{k_L}e^{2\pi i \langle k_L, x\rangle}: {k_L}\in L\}.
\end{equation*}
Where the $c_{k_L}$ are normalization constants such that $\lpnorm{e_{k_L}}{2}{\T{n}}=1$. Each of these vectors is an eigenvector of the Laplacian such that
\begin{equation*}
    -\Delta e_{k_L}(x) = 4\pi^2|k_L|^2e_{k_L}(x)\coloneq \la^2e_{k_L}(x).
\end{equation*}
We adopt the convection that $\T{n} = \R{n}/L^*$ as we will primarily be working with the lattice dual to $L^*$. We then have the following representation of $f\in L^2(\mathbb{T}^n)$ in terms of this basis.

\begin{equation*}
    f(x)=\sum_{{k_L}\in L} \hat{f}(k_L)e_{k_L}(x), \hspace{.1in}\hat{f}(k_L) = \int_{\mathbb{T}^n} f(x)e_{k_L}(-x)dx.
\end{equation*}
We shall also normalize so that $\lpn{f}{2}=1$. We define at the outset the following spectral projection operators for the operator $\sqrt{-\Delta}$,

\begin{align*}
    &P_{\la,\delta}f \coloneq \sum_{{k_L}\in A_{\la,\delta}} \hat{f}(k_L)e_{k_L}(x),
\end{align*}
where $A_{\lambda,\delta}$ is the annulus of radius 
$\lambda$ and width $2\delta$. Because of our normalization for Fourier series, we have abbreviated $2\pi k_L\in A_{\la,\delta}$ as $k_L\in A_{\la,\delta}$. Our goal is to prove bounds on these spectral projection operators in terms of $\lambda$ and $\delta$. The case of $\delta=1$ was settled by Sogge for all compact boundaryless manifolds, see \cite{S3} and \cite{S1}. That is to say width 1 projectors satisfy the universal estimates

\begin{equation*}
    \lpnorm{P_{\la,1}f}{p}{\T{n}} \leq C\la^{\mu(p)}\lpnorm{f}{2}{\T{n}},
\end{equation*}
where $\mu(p) = \max\{\frac{n-1}{2}(\frac{1}{2}-\frac{1}{p}), \frac{n-1}{2}-\frac{n}{p}\}$. Here, $C$ is a constant that does not depend on the spectral parameter $\la$ but may depend on other parameters such as $n, p,$ and $L^*$. The transition between these two regimes occurs at the critical exponent $p_c=\frac{2(n+1)}{n-1}$.

To gain improvements over these universal estimates we must consider shrinking windows, that is $\delta \rightarrow 0$ as $\lambda\rightarrow \infty$. The case $\delta=1$ is settled and the smallest width that will be considered is $\sim\lambda^{-1}$ as any smaller window results in the projection operator just being a projection onto an eigenspace for rational tori. Shrinking windows do not give improvements for all manifolds as $S^{n-1}$ follows the same bounds as the universal case $(\delta \sim 1)$ for all values of $\delta$. In the negative and non-positive sectional curvature cases gains are expected, see \cite{BHS1}, \cite{BS1}, and \cite{HS1} for results at this level of generality. In that setting, there is a natural barrier at windows of width $\sim(\log\la)^{-1}$. The fact that the universal covers of general manifolds have a number of Dirichlet domains that grows exponentially with respect to the radius of a ball centered at the origin make it difficult to analyze smaller windows with known methods. The torus case is different as its universal cover is a tiling of $\R{n}$ which clearly has a number of Dirichlet domains that grows polynomially with the radius of a ball centered at the origin. This allows us to access the conjectured bounds for the torus all the way down to the smallest reasonable window for rational tori $\sim \la^{-1}$.

The conjecture we are interested in is the following, originally stated by Germain--Myerson and partially proved in \cite{GM1}. We state it in its full level of generality for reference.

\begin{conjecture}[The Germain--Myerson conjecture, spectral projection bounds for the tours]\label{conj}
Let $\delta > \la^{-1+\kappa}$ for some fixed $\kappa \in (0,1]$. Then for some constant $C = C(n,p,\kappa, L^*)$ the spectral projection operators for $\T{n}$ obey the following estimates

    \begin{equation}
        \lpnb{P_{\la,\delta}}{2\rightarrow p} \leq \begin{cases}
            C(n,p,\kappa, L^*)(\la \delta)^{\frac{n-1}{2}(\frac{1}{2}-\frac{1}{p})} = C(\la \delta)^{\mu_1(p)} & p\leq p_c\\
            C(n,p,\kappa, L^*)(\la \delta)^{\frac{n-1}{2}(\frac{1}{2}-\frac{1}{p})} = C(\la \delta)^{\mu_1(p)} & p_c \leq p\leq p^* \text{ and } \delta \leq \la^{e(p)}\\
           C(n,p,\kappa, L^*)\la^{(\frac{n-1}{2}-\frac{n}{p})}\delta^{1/2} = C\la ^{\mu_2(p)}\delta^{1/2} &p^*\leq p \text{ and } \delta \geq \la^{e(p)}.\\
        \end{cases}
    \end{equation}
    Where $p_c = \frac{2(n+1)}{n-1}$ is the critical exponent, $p^* = \frac{2n}{n-2}$, and $e(p) = \frac{n+1}{n-1}(\frac{\frac{1}{p}-\frac{1}{p_c}}{\frac{1}{p}-\frac{n-3}{2(n-1)}})$. We have labeled 

    \begin{equation*}
        \mu_1(p) = \frac{n-1}{2}(\frac{1}{2}-\frac{1}{p}),\hspace{.2 cm} \mu_2(p) = \frac{n-1}{2} - \frac{n}{p}.
    \end{equation*}
\end{conjecture}
We are not concerned with the optimal value of $C$, only in obtaining the correct dependence on $\la$ and $\delta$. When the torus in question is $\T{n} = \R{n}/\Z{n}$, one conjectures analogous bounds up to $\delta\sim \la^{-1}$. This is the discrete restriction conjecture studied extensively by Bourgain, Demeter, and others. The above conjecture can be viewed as a generalization of the universal estimates of Sogge and the discrete restriction conjecture.

Constructions of the examples that saturate the conjectured bounds and an explanation of the relevant constants are provided by Germain--Myerson in Section 3 of \cite{GM1}. They also demonstrate the conjectured bounds would be sharp if true.

Partial progress has been made on Conjecture \ref{conj}. In particular, Germain--Myerson proved the following 

\begin{theorem}[Spectral projection bounds for the torus with $\epsilon$-loss, \cite{GM1}]\label{GMResultThr}
Let $\mathbb{T}^n = \R{n}/L^*$. Let $\delta > \la^{-1+\kappa}$ for some fixed $\kappa \in (0,1]$. Then for any $\epsilon>0$ and some $C(\epsilon) = C(n, p_c, \kappa, L^*,\epsilon)$ we have

\begin{equation}\label{mainResult}
    \lpnb{P_{\la,\delta}}{2\rightarrow p_c} \leq C(n, p_c, \kappa, L^*,\epsilon) \la^{\epsilon}(\la\delta)^{\mu_1(p_c)}= C(\epsilon)\la^{\epsilon}(\la\delta)^{\frac{1}{p_c}},
\end{equation}
which after interpolation with the trivial $L^2\rightarrow L^2$ estimates yields

\begin{equation}\label{mainResult}
    \lpnb{P_{\la,\delta}}{2\rightarrow p} \leq C(n,p,\kappa, L^*,\epsilon) \la^\epsilon(\la\delta)^{\mu_1(p)}\hspace{.3 cm} 2\leq p\leq p_c,
\end{equation}
which verifies Conjecture \ref{conj} for $2\leq p \leq p_c$ \say{with $\epsilon$-loss}.
\end{theorem}
The removal of these sub-polynomial losses in this setting is the subject of this paper. In the case $\T{2}=\R{2}/\Z{2}$, Demeter--Germain in \cite{DG1} removed these $\epsilon$-losses to prove sharp results when $2\leq p < p_c$.

Progress was also made when $p>p_c$ for some region of $(p,\delta)$ by Germain--Myerson in all dimensions in \cite{GM1}. These results took inspiration from and improved over the results of Hickman \cite{H1}. We defer to those papers for a full account.

All of these approaches critically use the decoupling theorem of Bourgain and Demeter which incurs a $O_\epsilon(\la^{\epsilon})$ loss at the critical exponent $p_c$. This loss is necessary when the decoupling theorem is applied as was proven by Bourgain in \cite{B1}, and so a sharp version of Theorem \ref{GMResultThr} cannot come from an abstract sharpening of the decoupling theorem. The conjecture at $p=p_c$ without epsilon loss was proved for windows of width $\delta > \la^{-\frac{1}{n+3}}$ by Blair--Huang--Sogge in \cite{BHS1} using methods applicable to more general manifolds that do not include technology such as decoupling. The general strategy of this paper will be to combine the techniques of Blair--Huang--Sogge with the techniques of Hickman, Germain--Myerson, and Demeter--Germain. This leads to the following sharp result.

\begin{theorem}[Spectral projection bounds for the torus]\label{MainResultThr}
Let $\mathbb{T}^n = \R{n}/L^*$. Let $\delta > \la^{-1+\kappa}$ for some fixed $\kappa \in (0,1]$. Then for some $C= C(n,p_c,\kappa, L^*)$ we have

\begin{equation}\label{mainResult}
    \lpnb{P_{\la,\delta}}{2\rightarrow p_c} \leq C(n,p_c,\kappa, L^*) (\la\delta)^{\mu_1(p_c)}= C(\la\delta)^{\frac{1}{p_c}},
\end{equation}
which after interpolation with the trivial $L^2\rightarrow L^2$ estimates yields

\begin{equation}\label{mainResult}
    \lpnb{P_{\la,\delta}}{2\rightarrow p} \leq  C(n,p,\kappa, L^*)(\la\delta)^{\mu_1(p)},\hspace{.3 cm} 2\leq p\leq p_c.
\end{equation}
This verifies Conjecture \ref{conj} for $2\leq p \leq p_c$ (without an $\epsilon$-loss).
\end{theorem}

\begin{remark}
    $p_c = \frac{2(n+1)}{n-1}$ is referred to as the critical exponent in both decoupling theory and in the Euclidean analog of this problem, but it is not strictly speaking a critical exponent for the Germain--Myerson conjecture as the full result does not follow from the result at $p_c$. Instead we have a critical curve in the $(\delta,p)$ plane. When $\delta = 1$ the corresponding critical value is $p_c$ but when $\delta\sim \la^{-1}$ the corresponding critical value is $p^* = \frac{2n}{n-2}$. Still, we will call $p_c$ the critical exponent as it bears this name in other problems. We note that this value of $p$ is also called the Stein-Tomas exponent in the context of restriction theory.
\end{remark}

\subsection{Notation}\label{notation}
We will write $A \lesssim B$ to mean that $A \leq C B$ where $C$ is a constant. For this paper, we will always allow the implicit constant to depend on $n, p,\kappa$ and the lattice $L^*$. For this problem $\kappa$ is fixed and our final estimates will not be defined as $\kappa$ approaches $0$. In particular, this approach cannot give sharp eigenfunction bounds for $\T{n} = \R{n}/\Z{n}$ which corresponds to $\kappa=0$. It is important that all implicit constants do not depend on the spectral parameter $\la$. For example, we may rewrite the main result as 

\begin{equation*}
    \lpnb{P_{\la,\delta}}{2\rightarrow p} \lesssim (\la\delta)^{\mu_1(p)},\hspace{.3 cm} 2\leq p\leq p_c.
\end{equation*}

If we wish to highlight dependence we will write $A\lesssim_\epsilon B$ to mean that $A\leq C(\epsilon) B$. The statement \say{with $\epsilon$ loss} will refer to $A\lesssim_\epsilon \la^\epsilon B$ for arbitrarily small but fixed choice of $\epsilon>0$. The smaller the choice of $\epsilon$, the larger the multiplied constant that depends on $\epsilon$. These are also referred to as \say{sub-polynomial losses} as the dependence in $\la$ must grow smaller than any polynomial. The expression $A\leq \ln(\la) B$ is an example of a bound that exhibits sub-polynomial losses. We will not worry about explicit bounds on the sub-polynomial factor. It is only relevant that such a factor grows slower than any fixed power of $\la$. We write $A\sim B$ to mean $A\lesssim B$ and $B\lesssim A$ for the same list of parameters. Writing a quantity is $\sim A$ means equality up to a choice of multiplied constant.

We use the notation $2^j \ll 1$ to indicate that the value of the dyadic sequence $\{2^j\}$ should not exceed some fixed absolute constant that depends on the dimension $n$.

We say that an object $A$ has rapid decay if $A\leq C(N)\la^{-N}$ for each natural number $N$. Such terms in expressions are often unimportant error terms as by selecting $N$ large the object $A$ will overpower any polynomial powers of $\la$.

When we state that a parameter or amount is $O(1)$, we always do so with respect to the spectral parameter $\la$. Such constants are allowed to depend on the parameters we hide with $\lesssim$. Throughout this paper, $C$ will be used to denote fixed constants whose specific values we wish to suppress. The value of $C$ will be allowed change from line to line but will never depend on $\la$.

\subsection{An overview of the proof}
This paper largely follows the work of Blair--Huang--Sogge \cite{BHS1} which began by approximating the spectral projection operator with a smoothed out version. This argument is standard. Microlocal cutoffs $Q_\nu$ were then introduced. The phase support of the kernels $Q_\nu(x,y)$ were contained in the intersection of an annulus of radius and width $\sim \la$ and a cone with aperture $\delta$ centered at the origin pointing in the direction $\nu$, where $\nu$ comes from a $\delta$ separated set in $S^{n-1}$. These operators satisfy estimates that align with the conjecture when composed with $P_{\la,1}$ and can be inserted into the problem with acceptable losses.

Next, a bilinear decomposition in terms of the $Q_\nu$ was deployed. Terms in this decomposition were sorted into two operators: one collected the diagonal terms where the supports of the kernels were close, and the other collected the far terms where the supports of the kernels were sufficiently separated. This can be encoded in the separation of the $\nu$ directions. Techniques developed in that paper were able to handle both the diagonal terms and the far terms in the context of general manifolds of non-positive sectional curvature, and sharp results were obtained.

There are two major departures in this paper from the work of Blair--Huang--Sogge needed to get sharp results in the torus case. First, we employ a bilinear decomposition of the discrete projection operator and only introduce the microlocal cutoffs afterwards. This sidesteps the use of estimates in the original paper that did not present issues with the larger windows under consideration, but are problematic with our target width of $\la^{-1+\kappa}$. Second, we will handle the off diagonal terms using decoupling, a cap counting result, and an explicit computation using Fourier series, tools which are not available in the more general case and require our projection operator to be written as a Fourier multiplier. This is another motivation for doing the bilinear decomposition before introducing microlocal cutoffs.

The diagonal terms are handled much in the same way as \cite{BHS1}, but are easier in this setting because of the added flexibility of working with torus multipliers. Whereas the estimate proved in that paper needed to be an $l^{p_c}$ norm, we can instead use the larger $l^2$ norm and orthogonality.

\vspace{.1in}

\textbf{Acknowledgement. }The author would like to thank Christopher D. Sogge, Xiaoqi Huang, and Connor Quinn for invaluable feedback, input, and support for the duration of this project. The author is also grateful to the two anonymous referees whose comments greatly improved the presentation of this paper. The author was supported in part by NSF Grant DMS-2348996.

\section{A decomposition}
We follow the setup in \cite{BHS1} but instead partition the Fourier multiplier before any approximation. Our decomposition into multipliers will be analogous to our decomposition that defines the microlocal cutoffs $Q_\nu$ that will be used later.

Let $\theta_0 = \delta$. This will always be the value of $\theta_0$ but we distinguish the two to track which contributions come from the angular decomposition and which come from the window, to maintain consistency with previous work, and to clarify matters when we eventually define $\theta_k = 2^k\theta_0$. Let $\{\nu\}$ be a maximal $\theta_0$-separated set in $S^{n-1}$ such that every $\xi\in S^{n-1}$ is within $\theta_0$ of a $\nu$. This collection of points will also be used to define the $Q_\nu$. For a fixed $\nu$, let $a_\nu$ be the set of $\xi\in S^{n-1}$ such that $|\xi-\nu|< |\xi-\nu'|$ for all $\nu\neq \nu'$. As written there may be a  set of points such that $|\nu-\xi| = |\nu'-\xi|$ for $\nu\neq \nu'$. We remove such a $\xi$ from one of $a_\nu,a_{\nu'}$ in this situation so that $\{a_\nu\}$ is a partition of $S^{n-1}$. Define  

\begin{equation*}
    A_{\la,\delta}^\nu \coloneq \{\xi\in \R{n}: |
    \xi|\in (\la-\delta,\la+\delta), \xi/|\xi|\in a_\nu\}.
\end{equation*}
It suffices to prove the results in this paper with the $\nu$ restricted to a sufficiently small $O(1)$ neighborhood of $(0,...,0,1)$. This follows as any rotated lattice is again a lattice, so the problem has rotational symmetry, and the triangle inequality. This reduction is to allow us to organize a sum based on the projection of $\nu$ to $\{x: x_n=0\}$ in Section 4.

Define

\begin{equation}
    P_{\nu}f \coloneq \sum_{{k_L}\in A^\nu_{\la,\delta}}\hat{f}(k_L) e_{k_L}(x),
\end{equation}
which allows us to make the following organization

\begin{equation}
    (P_{\la,\delta}f)^2 = \sum_{(\nu,\nu')\in \Xi_{\theta_0}}P_{\nu}fP_{\nu'}f + \sum_{(\nu,\nu')\notin \Xi_{\theta_0}}P_{\nu}fP_{\nu'}f = \Upsilon^{\text{diag}}(f) + \Upsilon^{\text{far}}(f),
\end{equation}
where $(\nu,\nu')\in \Xi_{\theta_0}$ if and only if $|\nu-\nu'| < 2^r \theta_0 = C\theta_0$ for some fixed small integer $r$. Here $r$ defines to threshold between the diagonal and far terms. Its exact value is not relevant as long as it is fixed and larger than say 2. Our goal now is to prove the following theorem.

\begin{theorem}[Spectral projection operators for the torus at $p_c$] Let $\delta > \la^{-1+\kappa}$ for some $\kappa\in (0,1]$. Let $\mathbb{T}^n = \R{n}/L^*$. Then 

\begin{equation}\label{thToPrv}
    \lpn{P_{\la,\delta}f}{L^{p_c}(\mathbb{T}^n)}\leq \lpn{(\Upsilon^{diag}f)^{1/2}}{{L^{p_c}(\mathbb{T}^n)}} +  \lpn{(\Upsilon^{\text{far}}f)^{1/2}}{{L^{p_c}(\mathbb{T}^n)}} \lesssim (\la\delta)^{\frac{1}{p_c}}\lpn{f}{{L^{2}(\mathbb{T}^n)}}.
\end{equation}
\end{theorem}
The first inequality follows from our previous discussion and the triangle inequality. Interpolation with the trivial estimates shows that this implies Theorem 1.3. The bound for the far terms is the content of Proposition \ref{farGoalProp}. Section \ref{capCountingSection} will establish some preparatory results to aid in that proof. The diagonal terms will be handled by Proposition \ref{diagGoalProp} which relies on microlocal estimates which are proved in Section \ref{microProofsSection}. 

\section{Cap Counting}\label{capCountingSection}
To prove their result with $\epsilon$-loss, Germain--Myerson used the decoupling theorem of Bourgain--Demeter which necessitates covering the annulus $A_{\la,\delta}$ with caps of a specific geometry. Decoupling will also be used in our proof. In this section we shall establish the relevant properties of the caps, present an abridged version of the proof with $\epsilon$-loss using decoupling, and discuss how decoupling will be utilized in our proof of sharp results.

\subsection{Understanding caps}
In this context, a \textbf{cap} will refer to a box in $\R{n}$ whose dimensions are given by $c_1(\la\delta)^{1/2}\times ... \times c_{n-1}(\la\delta)^{1/2}\times c_n \delta$. The constants $c_i$ will not depend on $\la$, will have no bearing on our arguments, and are oftentimes suppressed. Caps are allowed to have arbitrary center and orientation in $\R{n}$. More often in applications of decoupling caps are of the form $\sim R^{-1}\times ... \times R^{-1}\times R^{-2}$ for some large positive $R$. This discrepancy is a result of our scaling as we are considering $\la S^{n-1}$ not $S^{n-1}$. A suitable change of variables aligns the two perspectives.

Let $\Omega = \{\omega\}$ be a collection of finitely disjoint caps $\omega$ that cover the annulus $A_{\la,\delta}$. That is, $A_{\la,\delta}\subset \cup_{\omega\in \Omega} \omega$. Also, define

\begin{equation*}
    \#\omega \coloneq \#(\omega \cap 2\pi L).
\end{equation*}
This is the number of lattice point contained in a cap $\omega$. The factor of $2\pi$ is because of our normalization of Fourier series. When referring to lattice points in a cap we mean the lattice $2\pi L$. Later, we will relate the number of lattice points in a cap to the $L^p$ norm of $P_\omega f$, the Fourier projection onto frequencies contained in the cap $\omega$. As written the caps possibly overlap. To deal with this we employ a partition of unity. Let $\chi_\omega$ be smooth functions such that $0\leq \chi_\omega \leq 1$, $\supp{\chi_\omega}\subset \omega,$ and $\sum_{\omega\in \Omega} \chi_\omega = 1$ on $A_{\la,\delta}$. Initially define 

\begin{equation}
    P_{\omega}'f \coloneq \sum_{k_L\in \omega}\chi_\omega(2\pi k_L)\hat{f}(k_L)e_{k_L}(x),
\end{equation}
and

\begin{equation}\label{PomegeDefEqn}
    P_{\omega}f \coloneq P_{\la,\delta}P_{\omega}'f.
\end{equation}
We clearly have $P_{\la,\delta} = \sum_{\omega\in \omega}P_{\omega}$.

While we include the projection $P_{\la,\delta}$ in the definition of $P_{\omega}$ to make this a true decomposition of our operator, it will often be more productive to think of $P_\omega f$ as having Fourier support in the lattice points contained in $\omega$ not $\omega\cap A_{\la,\delta}$. The reason for this is the convexity of the caps $\omega$ and the easier geometric interpretation of it being a box in $\R{n}$. In doing so we will discuss estimates for lattice points in the cap $\omega$ which is clearly an upper bound for the number of lattice points in $\omega\cap A_{\la,\delta}$. As our estimates involving the $L^p$ norm of $P_\omega f$ will be bounded by the total number of lattice points in its Fourier support so bounding by the number of lattice points in $\omega$ instead leads to worse results that will still be sufficient.

Returning to the properties of the caps, the lattice points in $\omega$ have a dimension as a vector space (after shifting the cap to the origin, consider the span over $\R{}$ of the lattice points contained in the cap). We will denote this dimension $d_\omega$. This is important as by the log convexity of $L^p$ norms we have the following inequality for caps $\omega$.

\begin{equation}\label{numberCapBound}
    \lpnorm{P_\omega f}{p}{\T{n}}\lesssim (\#\omega)^{\frac{1}{2}-\frac{1}{p}}\lpnorm{P_\omega f}{2}{\T{n}},\, p\geq 2.
\end{equation}
Here we used the trivial $L^2\rightarrow L^\infty$ and $L^2\rightarrow L^2$ bounds. Recall that $\omega$ has a short side of length $\delta$. This means that if $d_\omega = n$ we have $\#\omega \lesssim \delta(\la\delta)^{\frac{n-1}{2}}$ which is the volume of the cap. We expect this bound to hold for most caps.

If $d_\omega = n-1$, we may have the worst possible bound of  $\#\omega\lesssim (\la\delta)^{\frac{n-1}{2}}$. Additionally, for such an $\omega$, one may check that using \eqref{numberCapBound} with this bound on the number of lattice points gives

\begin{equation}\label{maxCapBound}
    \lpnorm{P_\omega f}{p}{\T{n}}\lesssim (\la\delta)^{\frac{n-1}{2}(\frac{1}{2}-\frac{1}{p})}\lpnorm{P_\omega f}{2}{\T{n}},\, p\geq 2.
\end{equation}
Which aligns with the conjectured bound at $p_c$. The saturating example, the Knapp example in this context, is a function whose frequency support is contained in a single cap $\omega$ that contains the maximum number of lattice points. This is the \say{geodesic focusing case} as in physical space, such a function is concentrated in the neighborhood of a line which is a geodesic on the (flat) torus.

The result with an extra $O_\epsilon(\la^\epsilon)$ factor follows from the bound in \eqref{maxCapBound}, which clearly holds for any cap, after an application of the decoupling theorem. We shall later show that there are not many caps with the maximum number of lattice points. Despite using the bound $\#\omega\lesssim (\la\delta)^{\frac{n-1}{2}}$ for all caps in the decoupling argument, the only loss comes from decoupling itself because such a bound is the saturating example and because of orthogonality. Let us see how decoupling \say{almost} proves sharp results.

\subsection{The almost proof using decoupling} We will now introduce the decoupling theorem and show how it was used to prove the result with $\epsilon$-loss. 

\begin{theorem}[Discrete decoupling for the sphere, \cite{BD1}]\label{honestDecpl}
    Let $\Omega$ be a finitely disjoint covering of $\la S^{n-1}$ by $c_1(\la\delta)^{1/2}\times ... \times c_{n-1}(\la\delta)^{1/2} \times c_n\delta$ caps. Then for any 1-separated set $\xi_\alpha\subset \la S^{n-1}$ and sequence $\{a_\alpha\}$ it follows that for any fixed $\epsilon>0$,

    \begin{equation}
        \lpnorms{\sum_{\alpha}a_\alpha e(x\cdot \xi_\alpha)}{p_c}{B_{\delta^{-1}}}\lesssim_\epsilon \la^{\epsilon} \Bigl(\sum_{\omega\in\Omega} \lpnorms{\sum_{\alpha:\, \xi_\alpha \in \omega}a_\alpha e(x\cdot \xi_\alpha)}{p_c}{w_{B_{\delta^{-1}}}}^2\Bigr)^{1/2},
    \end{equation}
    where $B_{\delta^{-1}}$ is a spatial ball of radius $\delta^{-1}$ and $L^{p_c}(w_{B_{\delta^{-1}}})$ is a weighted norm.
\end{theorem}
It is known that there must be a $O_\epsilon(\la^\epsilon)$ loss in the above theorem, see \cite{D1} for a discussion of this and a precise definition of the weighted norm appearing on the right hand side. That the above theorem can be applied in our situation is known and was used by Hickman \cite{H1} and Germain--Myerson \cite{GM1} to make progress on Conjecture \ref{conj}.

The outline for applying this result to our situation is as follows. We will always take $\xi_\alpha$ to be coming from a lattice, which allows us to interpret the exponential sums as Fourier series. The localization at scale $\delta^{-1}$ provides uncertainty at scale $\delta$ in Fourier space, so we can consider points $\xi_\alpha$ that are contained in the annulus $A_{\la,\delta}$. This is done rigorously in the proof of Corollary 9 in \cite{H1}. Fourier series are periodic, and so the spatial norms can be replaced by an integral over the fundamental domain of the lattice which can be interpreted as $\T{n}$. The loss $\la^\epsilon$ ostensibly depends on $\la$ and $\delta$, but the exact power is subsumed by the $\epsilon$ exponent and so this dependency is suppressed. This justifies the following statement which is what we shall apply in our proofs.

\begin{proposition}\label{decpl}
    Let $\Omega$ be a collection of $c_1(\delta\la)^{1/2} \times ... \times c_{n-1}(\delta\la)^{1/2}\times c_n\delta$ finitely disjoint caps covering the annulus $A_{\la,\delta}$.  Let $P_{\la,\delta}$ be the Fourier projection operator onto $A_{\la,\delta}$ and $P_{\omega}$ be defined as in \eqref{PomegeDefEqn}. We then have

    \begin{equation}
        \lpnorm{P_{\la,\delta}f}{p_c}{\mathbb{T}^n} \lesssim_\epsilon \la^{\epsilon}\Bigl(\sum_{\omega\in \Omega}\lpnorm{P_{\omega}f}{p_c}{\mathbb{T}^n}^2 \Bigr)^{1/2}.
    \end{equation}
\end{proposition}
This is essentially the form reached in Corollary 9 in \cite{H1}. Combined with \eqref{numberCapBound} this immediately gives the result as

\begin{align}\label{decplGoodCapEqn}
    \begin{split}
    \lpnorm{P_{\la,\delta}f}{p_c}{\mathbb{T}^n} &\lesssim_\epsilon \la^{\epsilon}\Bigl(\sum_{\omega\in \Omega}\lpnorm{P_{\omega}f}{p_c}{\mathbb{T}^n}^2 \Bigr)^{1/2}\\
    &\lesssim_\epsilon \la^{\epsilon}\sup_{\omega}(\#\omega)^{\frac{1}{2}-\frac{1}{p_c}}\Bigl(\sum_{\omega\in \Omega}\lpnorm{P_{\omega}f}{2}{\mathbb{T}^n}^2 \Bigr)^{1/2}\\
    &\lesssim_\epsilon \la^\epsilon (\la\delta)^{\frac{n-1}{2}(\frac{1}{2}-\frac{1}{p_c})}\lpn{f}{2}.
    \end{split}
\end{align}
The above argument shows that while the majority of caps contain a number of lattice points less than $(\la\delta)^{\frac{n-1}{2}}$, the caps with a full number of lattice points dominate the norm. The above argument is also sharp up to the $O_\epsilon(\la^\epsilon)$ factor which comes from the application of decoupling \textit{not} the lattice point bound for an $\omega$. The $O_\epsilon(\la^\epsilon)$ is mandatory in the sense that an abstract application of decoupling must involve at least a $\log(\la)^c$ loss so the above argument cannot prove sharp bounds by an improved general decoupling estimate. 

Decoupling, however, will be useful for us because of its ability to handle a large number of caps. We will label caps that contain near the maximum number of lattice points \say{bad} and ones that contain less \say{good}. Specifically, for a fixed $\eta>0$, define

\begin{align}
    &C'_{\text{good}} = \{\omega: \#\omega \lesssim (\la\delta)^{\frac{n-1}{2}-\eta}\},\\
    &C'_{\text{bad}} = \{\omega: \#\omega \gtrsim (\la\delta)^{\frac{n-1}{2}-\eta}\}.
\end{align}

Define $P_{\text{good}}$ and $P_{\text{bad}}$ in an analogous way to $P_\omega$ using the same partition of unity. Clearly

\begin{equation}
    P_{\la,\delta} = P_{\text{good}} + P_{\text{bad}}.
\end{equation}
This is motivated the above argument as all good caps can be handled by an application of decoupling. We have by \eqref{decplGoodCapEqn}

\begin{equation}\label{goodCapBoundEqn}
    \lpnorm{P_{\text{good}}f}{p_c}{\mathbb{T}^n}\lesssim_\epsilon\la^{\epsilon} (\la\delta)^{(\frac{n-1}{2}-\eta)(\frac{1}{2}-\frac{1}{p_c})}\lpn{f}{2}.
\end{equation}
As $-\eta(\frac{1}{2}-\frac{1}{p_c})$ is negative for any choice of positive $\eta$ we can select $\epsilon$ small enough so that 

\begin{equation}
    \lpnorm{P_{\text{good}}f}{p_c}{\mathbb{T}^n}\lesssim_\eta (\la\delta)^{\frac{n-1}{2}(\frac{1}{2}-\frac{1}{p_c})}\lpn{f}{2}.
\end{equation}
And so decoupling can be used to show the good caps obey sharp bounds. This handles the vast majority of caps. The rest of our argument will be to bound the projection onto the bad caps. To do so we will need a quantitative bound on the number of such caps which is the subject of the next subsection and will the bad caps to be handled in Section \ref{farTermsSection}.

\begin{remark}
    The above bound for the good caps is included for illustrative purposes. In our full argument we employ the decomposition described in Section 2. We will then only use the this section's results involving caps on the far terms. The argument also works by first estimating the good caps as has been described, then using a bilinear decomposition only on the remaining bad caps. While this may lead to a cleaner exposition, it is also strictly weaker and so we have chosen to argue in this order.
\end{remark}

\subsection{A result about bad caps}
Our goal is to bound the number of bad caps. The threshold for a cap to be considered bad is controlled by the parameter $\eta$ which we may select freely as long as it is positive and does not depend on $\la$. It is allowed to depend on other parameters in the problem such as $n$ and $L^*$.

As the caps $\omega$ have $n-1$ long sides of length $(\la\delta)^{\frac{n-1}{2}}$ and $1$ short side of length $\delta$, the caps such that $d_\omega = n$ cannot have near the maximum number of lattice points. The caps such that $d_\omega \leq n-2$ obviously do not near the maximum number of lattice points as the separation distance between any two points is at least $O(1)$. Therefore, by a sufficiently small choice of $\eta$, every bad cap has $d_\omega=n-1$.

To count these caps we will use Theorem 4.1 from \cite{GM1}. In that paper, Germain--Myerson chose to analyze the equivalent problem of projecting onto a $\delta$ neighborhood of a quadratic form $Q(x)$ with the lattice fixed to be $\Z{n}$. We fix our quadratic form to be given by the sphere and let the lattice $L$ vary. This is necessitated by the use of the microlocal operators later in the paper. The following proposition states that the number of bad caps can be made small by taking $\eta$ to be sufficiently small.

We will prove the following theorem.

\begin{theorem}\label{badCapBoundThrm}
    Let $C_{\text{bad}}$ be the collection of all caps such that $\#\omega\gtrsim (\la\delta)^{\frac{n-1}{2}-\eta}$ with $\eta$ chosen small enough such that $d_\omega = n-1$ for every $\omega\in C_{\text{bad}}$. We then have

    \begin{equation}\label{badCapBoundEqn}
        \#C_{\text{bad}}\lesssim (\la\delta)^{\eta n}.
    \end{equation}
\end{theorem}
This is essentially the same result as Germain--Myerson Theorem 4.1 and our proof assuming that theorem is largely a change of variables. We recall Theorem 4.1 now, with some slight alterations. In that paper different notation and conventions were used. We have translated their result to our notation but the content is the same. Additionally we will restrict to the range of $\eta$ such that $d_\omega = n-1$ for every bad $\omega$ which allows for the following simpler version of Theorem 4.1.\footnote{Specifically, \cite{GM1} indexed their collection of caps by $2^j$ by gathering caps containing $\sim (\la\delta)^{\frac{n-1}{2}}\delta 2^j$ lattice points. For the theorem to hold as stated we must take $2^j> K\delta^{-1}(\la\delta)^{-1/2}$ where $K$ is an absolute constant from that paper. For the same condition to be satisfied by our bad caps we need $\delta^{-1}(\la\delta)^{-\eta}>2^j> K\delta^{-1}(\la\delta)^{-1/2}$ or $\eta<\frac{1}{2}$. This is our first condition on $\eta$.} We have

\begin{lemma}[Theorem 4.1, \cite{GM1}]\label{GMbadCapthrm}
    Let $Q$ be a quadratic form and $S_{\la,\delta} =\{x\in \R{n}: |\sqrt{Q(x)}-\la|<\delta\}$. Let $\theta$ be almost disjoint caps intersected with $S_{\la,\delta}$ such that $S_{\la,\delta} = \cup_{\theta}\theta$. Let $C^Q_\text{bad} = \{\theta: \#\theta \gtrsim (\la\delta)^{\frac{n-1}{2}-\eta}\}$ for a sufficiently small $\eta$. Then

    \begin{equation}
        C^Q_\text{bad} \lesssim (\la\delta)^{\eta n}.
    \end{equation}
\end{lemma}
\begin{proof}[Proof of Theorem \ref{badCapBoundThrm} assuming Lemma \ref{GMbadCapthrm}] The $\theta$ are not caps by our convention but caps intersected with $S_{\la\,\delta}$. Let $M:\R{n}\rightarrow \R{n}$ be the invertible linear transformation such that $M(L) = \Z{n}$. We apply the change of variables given by $M(x_1,...,x_n) = (s_1,...,s_n)$.

In the $s$-variables our lattice is given by $\Z{n}$ and $A_{\la,\delta}$ is  the $c_M\delta$ neighborhood of  $\sqrt{Q}$ for some quadratic form $Q$. The constant $c_M$ in this instance comes from  $M$ which of course depends only on $L$ and not $\la$. Constants throughout this proof will share this dependence.

We have that $A_{\la,\delta}\subset \cup \omega$. After changing coordinates $\omega$ is not necessarily a cap in the $s$-variables. However, each $\omega$ must fit in a cap in the $s$-variables of dimensions $c_{1,M}(\la\delta)^{1/2}\times ... \times c_{n-1,M}(\la\delta)^{1/2}\times c_{n,M}\delta$. For each $\omega$ we will call $\theta'(\omega)$ an $s$-cap that contains it. This is because $M$ is linear.

Clearly the collection of $S_{\la,\delta}\subset \theta'(\omega)$ where $S_{\la,\delta}$ is associated to the quadratic form $Q$ coming from viewing $A_{\la,\delta}$ in the $s$-variables. Write $\theta(\omega) = S_{\la,\delta}\cap \theta'(\omega)$. The $\omega$ are almost disjoint which implies the $\theta(\omega)$ can be taken to be as well. The collection $\theta(\omega)$ satisfy the conditions of Lemma \ref{GMbadCapthrm}. Therefore

\begin{equation}
    C^Q_{\text{bad}}\lesssim (\la\delta)^{\eta n}.
\end{equation}
Also note that $\#\omega\leq \#\theta(\omega)$ and so $\#C_{\text{bad}}\leq \#C^Q_{\text{bad}}$. This gives the conclusion.
\end{proof}
\section{Handling the far terms}\label{farTermsSection}
\subsection{A reduction}
We first wish to prove a bound like \eqref{mainResult} for $\Upsilon^{\text{far}}$.

\begin{proposition}\label{farGoalProp}
    \[\lpn{(\Upsilon^{\text{far}}f)^{1/2}}{p_c} = \lpn{(\sum_{(\nu,\nu')\notin \Xi_{\theta_0}}P_{\nu}fP_{\nu'}f)^{\frac{1}{2}}}{p_c} = \lpn{(\sum_{(\nu,\nu')\notin \Xi_{\theta_0}}P_{\nu}fP_{\nu'}f)}{p_c/2}^{1/2} \lesssim (\la \delta)^{\frac{1}{p_c}}\lpn{f}{2}.\]
\end{proposition}
The diagonal terms will be handled in the next section. We shall organize our sum using a Whitney decomposition as was done successfully in Section 4 of \cite{BHS1}. Here we use our reduction to directions close to $(0,...,0,1)$.

We organize based on the first $n-1$ coordinates. Consider dyadic cubes in $\{x_n=0\}\cong \R{n-1}$ of sidelength $\theta_k = \theta_0 2^k$. We will label $\tau^{\theta_k}_{\mu}$ as the translations of $[0,\theta_k)^{n-1}$ by $\mu\in \theta_k\Z{n-1}$. 

We do not have to worry about $\mu,\mu'$ that are close by the definition of $\Xi_{\theta_0}$, so we will instead worry about organizing based on distance. We will call two cubes close if they have sidelength $\theta_k$, are not adjacent, and are contained in adjacent cubes of sidelength $2\theta_k$. Clearly there are $O(1)$ close pairs for any fixed cube, and all close cubes are separated by about $\theta_k$. We write $\tau^{\theta_k}_{\mu}\sim \tau^{\theta_k}_{\mu'}$ if two cubes $\tau^{\theta_k}_{\mu}, \tau^{\theta_k}_{\mu'}$ of sidelength $\theta_k$ are close.

We will write $\nu\in \tau_{\mu}^{\theta_k}$ if the projection of $\nu$ onto the first $n-1$ coordinates is within $\tau_{\mu}^{\theta_k}$. Given $(\nu,\nu')\not\in \Xi_{\theta_0}$ there is a unique set $(\theta_k,\tau^{\theta_k}_{\mu}, \tau^{\theta_k}_{\mu'})$ such that $\nu\in \tau^{\theta_k}_{\mu}, \nu'\in \tau^{\theta_k}_{\mu'}$, and $\tau^{\theta_k}_{\mu}\sim \tau^{\theta_k}_{\mu'}$. Our reduction to directions near $(0,...,0,1)$ ensures that projections of $(\nu,\nu')$ onto the first $n-1$ coordinates being close implies that $(\nu,\nu')$ are close as points on $S^{n-1}$.

All of this allows for the organization

\begin{equation}
    \sum_{(\nu,\nu')\notin \Xi_{\theta_0}}(P_{\nu}f)(P_{\nu'}f) = \sum_{\{k\geq r: 2^k\theta_0 \ll 1\}}\sum_{(\mu,\mu'): \tau^{\theta_k}_{\mu}\sim \tau_{\mu'}^{\theta_k}}\sum_{(\nu,\nu')\in \tau_{\mu}^{\theta_k}\times \tau_{\mu'}^{\theta_k}}(P_{\nu}f)(P_{\nu'}f).
\end{equation}
Where $r$ is such that $|\nu-\nu'|< 2^r\theta_0\implies (\nu,\nu')\in \Xi_{\theta_0}$ and the maximum value of $k$ is the smallest such that $\sup_{\nu,\nu'}|\nu-\nu'|<2^k\theta_0$. As $\nu,\nu'$ come from the sphere this quantity is $O(1)$. The advantage of this organization is that we have isolated the terms that contribute. The first sum contains $\sim \log(\la)$ terms, while the second sum has the property that for a fixed $\mu$ there exist only $O(1)$ values of $\mu'$ such that the relation $\tau^{\theta_k}_{\mu}\sim \tau_{\mu'}^{\theta_k}$ holds. This allows us to reduce to the following proposition.

\begin{proposition}\label{arcResult}
    Let $P_{\gamma}, P_{\gamma'}$ be projections onto disjoint angular sectors of $A_{\la,\delta}$ that are subsets of a fixed  neighborhood of $(0,...,0,1)$. Let each be defined by aperture $\sim \theta_k$ and have angular separation $\sim\theta_k$ from each other. There exists a fixed positive number $\sigma$ such that the following holds.
    \begin{equation}\label{normBoundGoal}
        \lpn{(P_{\gamma}f P_{\gamma'}f)^{1/2}}{p_c}\lesssim (\la\delta)^{\frac{1}{p_c}-\sigma}\Bigl(\lpn{P_{\gamma}f}{2}\lpn{P_{\gamma'}f}{2}\Bigr)^{1/2}.
    \end{equation}
\end{proposition}
Before we prove this, let us show that it implies Proposition \ref{farGoalProp}.
\begin{proof}[Proof that Proposition \ref{arcResult}$\implies$ Proposition \ref{farGoalProp}]
Note that by our restriction using the triangle inequality and rotation, the following equality holds

\begin{equation}
    P_{\gamma}\Bigl(\sum_{\nu\in \tau^{\theta_k}_{\mu}}P_{\nu}f\Bigr) = \sum_{\nu\in \tau^{\theta_k}_{\mu}}P_{\nu}f,
\end{equation}
if we widen the definition of $P_{\gamma}$ by a constant that does not depend on $\la$. This allows us to apply the following inequalities assuming Proposition \ref{arcResult}.

    \begin{align*}
        &\lpn{(\sum_{(\nu,\nu')\notin \Xi_{\theta_0}}P_{\nu}fP_{\nu'}f)^{\frac{1}{2}}}{p_c}\\
        &\leq\Bigl(\sum_{\{k\geq r: 2^k\theta_0 \ll 1\}}\sum_{(\mu,\mu'): \tau^{\theta_k}_{\mu}\sim \tau_{\mu'}^{\theta_k}}\lpn{\sum_{(\nu,\nu')\in \tau_{\mu}^{\theta_k}\times \tau_{\mu'}^{\theta_k}}P_{\nu}fP_{\nu'}f}{p_c/2}\Bigr)^{1/2}\\ 
        &=\Bigl(\sum_{\{k\geq r: 2^k\theta_0 \ll 1\}}\sum_{(\mu,\mu'): \tau^{\theta_k}_{\mu}\sim \tau_{\mu'}^{\theta_k}}\lpn{P_{\gamma}\Bigl(\sum_{\nu\in \tau^{\theta_k}_{\mu}}P_{\nu}f\Bigr)P_{\gamma'}\Bigl(\sum_{\nu'\in \tau^{\theta_k}_{\mu'}}P_{\nu'}f\Bigr)}{p_c/2}\Bigr)^{1/2}\\ 
        &\lesssim\Bigl(\sum_{\{k\geq r: 2^k\theta_0 \ll 1\}}\sum_{(\mu,\mu'): \tau^{\theta_k}_{\mu}\sim \tau_{\mu'}^{\theta_k}}(\la\delta)^{\frac{2}{p_c}-2\sigma}\lpn{\sum_{\nu\in \tau^{\theta_k}_{\mu}}P_{\nu}f}{2} \lpn{\sum_{\nu'\in \tau^{\theta_k}_{\mu'}}P_{\nu'}f}{2}\Bigr)^{1/2}\\
        &\lesssim (\la\delta)^{\frac{1}{p_c}-\sigma}\Bigl(\sum_{\{k\geq r: 2^k\theta_0 \ll 1\}}\sum_{(\mu,\mu'): \tau^{\theta_k}_{\mu}\sim \tau_{\mu'}^{\theta_k}}\Bigl(\sum_{\nu\in \tau^{\theta_k}_{\mu}}\lpn{P_{\nu} f}{2}^2\Bigr)^{1/2}\Bigl(\sum_{\nu'\in \tau^{\theta_k}_{\mu'}}\lpn{P_{\nu'} f}{2}^2\Bigr)^{1/2}\Bigr)^{1/2}\\
        &\lesssim (\la\delta)^{\frac{1}{p_c}-\sigma}\Bigl(\sum_{\{k\geq r: 2^k\theta_0 \ll 1\}}\Bigl(\sum_{(\mu,\mu'):\tau^{\theta_k}_{\mu}\sim \tau_{\mu'}^{\theta_k}}\sum_{\nu\in \tau^{\theta_k}_{\mu}}\lpn{P_{\nu} f}{2}^2\Bigr)^{1/2}\Bigl(\sum_{(\mu,\mu'): \tau^{\theta_k}_{\mu}\sim \tau_{\mu'}^{\theta_k}}\sum_{\nu'\in \tau^{\theta_k}_{\mu'}}\lpn{P_{\nu'} f}{2}^2\Bigr)^{1/2}\Bigr)^{1/2}\\
        &\lesssim (\la\delta)^{\frac{1}{p_c}-\sigma}\Bigl(\sum_{\{k\geq r: 2^k\theta_0 \ll 1\}}\sum_{\mu}\sum_{\nu\in \tau^{\theta_k}_\mu} 
        \lpn{P_{\nu}f}{2}^2\Bigr)^{1/2}\\
        &\lesssim \log(\la)(\la\delta)^{\frac{1}{p_c}-\sigma}\lpn{f}{2} \leq (\la\delta)^{\frac{1}{p_c}}\lpn{f}{2}.
    \end{align*}
Here we have used the fact that the first sum only has $\sim \log(\la)$ terms as it is a dyadic sum. This $\log$ loss is absorbed by our assumed power gain. We used that for a fixed $\mu$ there exist only $O(1)$ values of $ \mu'$ such that $\tau^{\theta_k}_{\mu}\sim \tau_{\mu'}^{\theta_k}$. This allows us to bound the double sum over ${(\mu,\mu'): \tau^{\theta_k}_\mu \sim \tau^{\theta_k}_{\mu'}}$ by a single sum over $\mu$ at the price of an $O(1)$ constant by the triangle inequality. We also used Cauchy-Schwarz and orthogonality.
\end{proof}

As we will handle the diagonal case by other means, we are free to exploit advantages present in the bilinear setup of Proposition \ref{arcResult}. We will only use decoupling on good caps as outlined in Section \ref{capCountingSection}. For the bad caps we will compute their bilinear interaction more directly to get a gain. Essentially, any two interacting bad caps must be sufficiently transverse so that their product has an improved norm.

Using the relation 

\begin{equation*}
    P_\gamma = P_\gamma P_{\la,\delta} = P_\gamma P_{\text{good}}+ P_\gamma P_{\text{bad}},
\end{equation*}
naturally define the projections $P^{\gamma}_{\text{good}}$ and $P^{\gamma}_{\text{bad}}$, as well as $C^{\gamma}_{\text{good}}$ and $C^{\gamma}_{\text{bad}}$. By the triangle inequality it suffices to bound the following by the bound in \eqref{normBoundGoal} to prove Proposition \ref{arcResult}.

\begin{equation}\label{goodBadtermsAfterTriangleEqn}
    \lpn{(P^{\gamma}_{\text{good}}f P^{\gamma'}_{\text{good}}f)^{1/2}}{p_c} + \lpn{(P^{\gamma}_{\text{good}} fP^{\gamma'}_{\text{bad}}f)^{1/2}}{p_c} + \lpn{(P^{\gamma}_{\text{bad}}f P^{\gamma'}_{\text{good}}f)^{1/2}}{p_c} + \lpn{(P^{\gamma}_{\text{bad}} fP^{\gamma'}_{\text{bad}}f)^{1/2}}{p_c}.
\end{equation}
It is important our bound does not depend on $\gamma,\gamma'$ and this will be apparent in our proof of the following proposition. We will bound the fourth term in the following subsection. By applying H\"older's in the below proof we are not taking advantage of any gains from the bilinear interaction of the first three terms. This is a manifestation of how decoupling handles all caps but the bad ones. Understanding these interactions better could be useful for future work.

\begin{proposition}\label{firstThreeTermsBoundsDecplProp}
    There exists positive number $\sigma_1$ that depends on $\eta$ (which, in turn, does not depend on $\la$) so that we have
    \begin{align}
    \begin{split}
        &\lpn{(P^{\gamma}_{\text{good}}f P^{\gamma'}_{\text{good}}f)^{1/2}}{p_c} + \lpn{(P^{\gamma}_{\text{good}} fP^{\gamma'}_{\text{bad}}f)^{1/2}}{p_c} + \lpn{(P^{\gamma}_{\text{bad}}f P^{\gamma'}_{\text{good}}f)^{1/2}}{p_c}\\
        &\lesssim (\la\delta)^{\frac{1}{p_c}-\sigma_1}\Bigl(\lpn{P_{\gamma}f}{2}\lpn{P_{\gamma'}f}{2}\Bigr)^{1/2}.
        \end{split}
    \end{align}
\end{proposition}
\begin{proof}
We prove the bound for one term as the proof is similar for all three. Using \eqref{goodCapBoundEqn} we have

\begin{align*}
    &\lpn{(P^{\gamma}_{\text{good}} fP^{\gamma'}_{\text{bad}}f)^{1/2}}{p_c}\\&\leq \Bigl(\lpn{P^{\gamma}_{\text{good}}f}{p_c}\lpn{P^{\gamma'}_{\text{bad}}f}{p_c}\Bigr)^{1/2}\\
    &\lesssim_{\epsilon}\la^\epsilon\Bigl(\sum_{\omega\in C^{\gamma}_{\text{good}}}\lpn{P^{\gamma}_\omega f}{p_c}^2\Bigr)^{1/4}\Bigl(\sum_{\omega\in C^{\gamma'}_{\text{bad}}}\lpn{P^{\gamma'}_\omega f}{p_c}^2\Bigr)^{1/4}  \\
    &\lesssim_\epsilon \la^\epsilon (\la\delta)^{\frac{n-1}{2}(\frac{1}{2}-\frac{1}{p_c})}(\la\delta)^{-\eta(\frac{1}{2}-\frac{1}{p_c})}\Bigl(\sum_{\omega\in C^{\gamma}_{\text{good}}}\lpn{P^{\gamma}_\omega f}{2}^2\sum_{\omega\in C^{\gamma'}_{\text{bad}}}\lpn{P^{\gamma'}_\omega f}{2}^2\Bigr)^{1/4}\\
    &\lesssim_\epsilon \la^{\epsilon}(\la\delta)^{\mu_1(p_c)}(\la\delta)^{-\eta(\frac{1}{2}-\frac{1}{p_c})}\Bigl(\lpn{P_\gamma f}{2}\lpn{P_{\gamma'}f}{2}\Bigr)^{1/2}\\
    &\lesssim_\eta (\la\delta)^{\mu_1(p_c)-\sigma_1}\Bigl(\lpn{P_\gamma f}{2}\lpn{P_{\gamma'}f}{2}\Bigr)^{1/2}
\end{align*}

Choosing $\epsilon$ such that $\la^\epsilon = (\la\delta)^{\frac{\eta}{2(n+1)}}$ allows us to conclude the last line with $\sigma_1 = \frac{\eta}{2(n+1)}$. The other cross term is handled similarly with the same bound and the first term in \eqref{goodBadtermsAfterTriangleEqn} satisfies better estimates. This bound gives an improvement for any choice of positive $\eta$ and so places no constraints on the parameter as we select $\epsilon$ once $\eta$ is fixed. This will change out multiplied constant but later we will select $\eta$ in a manner that depends on the other fixed parameters in the problem. 
\end{proof}

\subsection{Estimating the bad caps}
To estimate the bad caps, we will use an explicit computation of the $L^4$ norm. We will then interpolate with the $L^2\rightarrow L^\infty$ or $L^2\rightarrow L^2$ norm depending on $n$. 

The following general result connects the geometry of lattice points in caps to $L^4$ norms.

\begin{theorem}\label{fseries}
    Let $n\geq 2$ and $A,B$ be subsets of a lattice $L$. For $(a,b)\in A\times B$, define 
    \begin{equation*}
        S(a,b) \coloneq\{(a',b')\in A\times B: a + b = a' + b'\},
    \end{equation*}
    and suppose $S$ is such that $S(a,b)\leq S$ for every $a,b\in A\times B$. Then 

    \begin{equation*}
        \lpnorm{(P_Af P_Bf)^{1/2}}{4}{\mathbb{T}^n}\lesssim S^{1/4}\Bigl(\lpnorm{P_Af}{2}{\mathbb{T}^n}\lpnorm{P_Bf}{2}{\mathbb{T}^n} \Bigr)^{1/2}
    \end{equation*}
    where $P_A f$ and $P_B f$ are the Fourier projections of $f$ onto the set $A, B$ respectively.
\end{theorem}

\begin{proof}
    This is an explicit calculation. We have

    \begin{align*}
    &\lpn{(P_Af P_Bf)^{1/2}}{4}^4 = \int_{\mathbb{T}^n}|\sum_{(a,b)\in A \times B}\hat{f}(a)\hat{f}(b)e_{a+b}(x)|^2dx\\
    &=\int_{\T{n}}|\sum_{c\in A+B}\Bigl(\sum_{\substack{ a+b = c\\(a,b)\in A\times B}}\hat{f}(a)\hat{f}(b)\Bigr) e_c(x)|^2dx \\
    &= \sum_{c\in A+B}|\sum_{\substack{ a+b = c\\(a,b)\in A\times B}}\hat{f}(a)\hat{f}(b)|^2\\
    &\leq \sum_{(a,b)\in A \times B}|\hat{f}(a)\hat{f}(b)|\Bigr(\sum_{(a',b')\in S(a,b)}|\hat{f}(a')\hat{f}(b')|\Bigl)\\
    &\leq S^{1/2}\sum_{(a,b)\in A \times B}|\hat{f}(a)\hat{f}(b)|\Bigr(\sum_{(a',b')\in S(a,b)}|\hat{f}(a')\hat{f}(b')|^2\Bigl)^{1/2}\\
    &\leq S^{1/2}\Bigl(\sum_{(a,b)\in A \times B}|\hat{f}(a)\hat{f}(b)|^2\Bigr)^{1/2}\Bigr(\sum_{(a,b)\in A \times B}\sum_{(a',b')\in S(a,b)}|\hat{f}(a')\hat{f}(b')|^2\Bigl)^{1/2}\\
    &\leq S\sum_{(a,b)\in A \times B}|\hat{f}(a)\hat{f}(b)|^2\\
    &=S \bigl(\sum_{a\in A}|\hat{f}(a)|^2\bigr)\bigl(\sum_{b\in B}|\hat{f}(b)|^2 \bigr)\\
    &= S\lpn{P_Af}{2}^2\lpn{P_Bf}{2}^2.
\end{align*}
 Where we used Plancharel and Cauchy-Schwarz twice. Taking $4$-th roots allows us to conclude. 
\end{proof}
This result and proof is essentially a discrete version of Proposition 3.11 from \cite{D1}, although it applies in more generality than the neighborhoods of hypersurfaces considered there and without a transversality constant. It can also be viewed as a bilinear version of a well known connection between additive energies and the $L^4$ norm of Fourier series. See Chapter 4.5 of \cite{TV1} for an exploration of these connections.

\begin{proposition}\label{badbadprop}
    Let $P^{\gamma}_{\text{bad}}, P^{\gamma'}_{\text{bad}}$ be the projections onto the collection of bad caps intersected with the appropriate angular sectors as before. Then for a fixed positive number $\sigma_2$ we have
    
    \begin{equation}\lpnorm{(P^{\gamma}_{\text{bad}} P^{\gamma'}_{\text{bad}}f)^{1/2}}{p_c}{\mathbb{T}^n}\lesssim (\la\delta)^{\frac{1}{p_c}-\sigma_2}\bigl(\lpn{P_\gamma f}{2}\lpn{P_{\gamma '}f}{2}\Bigr)^{1/2}.\end{equation}
    Which also concludes the proof of Proposition \ref{arcResult} and therefore Proposition \ref{farGoalProp}.
\end{proposition}

\begin{proof}
    Select two caps, $\omega,\omega'$ such that $P^{\gamma}_{\text{bad}}P_\omega \neq 0$ and $P^{\gamma'}_{\text{bad}}P_{\omega'} \neq 0$. Because of our bound on the number of such $\omega,\omega'$ we can prove a power improvement for the interaction of these two caps and then use the triangle inequality to conclude. By Proposition \ref{fseries} we have 

    \begin{equation}
        \lpnorm{(P_\omega f P_{\omega'} f)^{1/2}}{4}{\T{n}}\lesssim S^{1/4}\Bigl(\lpn{P_{\omega} f}{2}\lpn{P_{\omega'}f}{2}\Bigr)^{1/2}.
    \end{equation}
    Where $S$ is the $\sup{}$over $(a,a')\in (\omega,\omega')$ of

    \begin{equation}\label{sidonSetBoundEqn}
        S(a,a') = |\{(b,b')\in (\omega,\omega'): a + a' = b + b'\}|.
    \end{equation}
     Fix $(a,a')\in (\omega,\omega')$. To bound this set we will only use the geometry of the caps themselves. Note that for a pair $(b,b')$ to contribute to \eqref{sidonSetBoundEqn} for a fixed $(a,a')$ we must have

    \begin{equation}
        b - a = b' - a'.
    \end{equation}
    Define $V = 2\pi L\cap\omega - a$ and $V' = 2\pi L\cap\omega-a'$. The carnality of $S(a,a')$ is the same as the cardinality of $V\cap V'$. 

    Note that $V$ and $V'$ are contained in $\sim (\la\delta)^{\frac{1}{2}}\times ... \times (\la\delta)^{\frac{1}{2}}\times \delta$ boxes in $\R{n}$ that both contain the origin. Denote $n_{V}$ and $n_{V'}$ to be the normal vectors of the hyperplane containing $V$ and $V'$ respectively. Our restriction on $\eta$ so that $d_\omega = n-1$ for every bad cap allows this to be well defined up to a choice of sign.
    
    As $\omega,\omega'$ come from arcs of an annulus with angular separation at least $\delta$ it must be that $n_V$ and $n_{V'}$ also have angular separation at least $\delta$.

    We will illustrate the calculation in $n=2$. Here, we are interested in the intersection of two rectangles that contain the origin that are at least $\delta$ transverse. This set is a parallelogram with height $\lesssim\delta$ and base $\lesssim\frac{\delta}{\sin(\nu)}$ where $\nu$ is the angle formed by $n_V$ and $n_{V'}$. The area can be calculated as 

    \begin{equation*}
        A \lesssim \frac{\delta^2}{\sin(\nu)}.
    \end{equation*}

    Using elementary geometry and the bound $\nu \gtrsim \delta$, we get that this parallelogram can be contained in a ball of $O(1)$ radius. In particular, any two points in the parallelogram are at most $O(1)$ separated. As our lattice points come from an $O_L(1)$ separated set, we have that the intersection contains less than $\lesssim 1$ lattice points.

    This represents a drop in dimension compared to an individual bad cap. In $n=2$ lattice points gather on a length $\sim(\la\delta)^{1/2}$ line in a cap $\omega$. With our bilinear interaction, the points that significantly contribute for a fixed $(a,a')$ must be contained in a $\lesssim(\la\delta)^0  \sim 1$ radius ball which forces the gain using Theorem \ref{fseries} as $S\lesssim 1$.

    To handle higher dimensions it is productive to think of $V$ and $V'$ as $\delta$ neighborhoods of subsets of $n-1$ dimensional hyperplanes. As the normals of these hyperplanes are separated their intersection is an $n-2$ dimensional linear space $W$. The intersection $V\cap V'$ is then the set of all lattice points that are $\lesssim \delta$ from the set $W$ in the direction $n_{V}$ and $\lesssim \delta$ from the set $W$ in the direction $n_{V'}$. Remembering the calculation in $n=2$, this is a subset of all of the lattice points within $O(1)$ of the set $W$. Here, we crucially use that the normals are at least $\delta$ separated. Because the lattice points of $L$ have $O(1)$ separation distance and any line segment contained in a cap has length $\lesssim (\la\delta)^{\frac{1}{2}}$ we have for arbitrary $n$

    \begin{equation}
        S(a,a')\lesssim (\la\delta)^{\frac{n-2}{2}}.
    \end{equation}
    For this bound we also used that for a fixed $(a,a')$ we are only allowed to select one point from $V\cap V'$ as this forces the choice of the other point. As this bound is uniform in $(a,a')$ we have $S\lesssim (\la\delta)^{\frac{n-2}{2}}$.
    Applying Theorem \ref{fseries} now gives us
    \begin{equation}
        \lpn{(P_\omega f P_{\omega'}f)^{1/2}}{4}\lesssim (\la\delta)^{\frac{n-2}{8}}\Bigl(\lpn{P_\omega f}{2}\lpn{P_{\omega'}f}{2}\Bigr)^{1/2}.
    \end{equation}
     If $n=2$, then we interpolate with $L^2\rightarrow L^\infty$ estimates. Specifically, as $\#\omega \lesssim (\la\delta)^{1/2}$ we have

     \begin{equation*}
         \lpn{(P_\omega f P_{\omega'}f)^{1/2}}{\infty}\lesssim (\la\delta)^{\frac{1}{4}}\Bigl(\lpn{P_\omega f}{2}\lpn{P_{\omega'}f}{2}\Bigr)^{1/2}.
     \end{equation*}
    Which after interpolation yields
    \begin{equation*}
        \lpnorm{(P_\omega f P_{\omega'}f)^{1/2}}{6}{\mathbb{T}^2}\lesssim (\la\delta)^{\frac{1}{12}}\Bigl(\lpn{P_\omega f}{2}\lpn{P_{\omega'}f}{2}\Bigr)^{1/2},
    \end{equation*}
    which is better than a $(\la\delta)^{\frac{1}{6}}$ bound. If $n=3$, then $p_c=4$ then this bound is sufficient as $\frac{1}{8}\leq \frac{1}{4}=\frac{1}{p_c}$. If $n\geq 4$ we interpolate with the trivial $L^2\rightarrow L^2$ estimates to get

    \begin{equation*}
        \lpnorm{(P_\omega f P_{\omega'}f)^{1/2}}{p_c}{\mathbb{T}^n}\lesssim (\la\delta)^{\frac{n-2}{2(n+1)}}\Bigl(\lpn{P_\omega f}{2}\lpn{P_{\omega'}f}{2}\Bigr)^{1/2},
    \end{equation*}
    which is better than $(\la\delta)^{\frac{n-1}{2(n+1)}}$. We can conclude with the triangle inequality and $\#C_{\text{bad}} \lesssim (\la\delta)^{\eta n}$ supposing, as we may, that $\eta$ is small enough. We will restrict to $n\geq 4$ but $n=2,3$ follow similarly. Write

    \begin{align*}
        \lpnorm{(P^{\gamma}_{\text{bad}} P^{\gamma'}_{\text{bad}}f)^{1/2}}{p_c}{\mathbb{T}^n}&=\lpnorm{(\sum_{\omega\in C_{\text{bad}}^\gamma} P_{\omega}f\sum_{\omega'\in C_{\text{bad}}^{\gamma'}} P_{\omega'}f)^{1/2}}{p_c}{\mathbb{T}^n} \\
        &\leq \bigl(\sum_{(\omega,\omega')\in C_{\text{bad}}^\gamma\times C_{\text{bad}}^{\gamma'}}\lpnorm{P_\omega f P_{\omega'}f}{p_c/2}{\T{n}}\Bigr)^{1/2}\\
        &\lesssim(\la\delta)^{\frac{n-2}{2(n+1)}}\Bigl(\sum_{(\omega,\omega')\in C_{\text{bad}}^\gamma\times C_{\text{bad}}^{\gamma'}}\lpn{P_\omega f}{2}\lpn{P_{\omega'}f}{2}\Bigr)^{1/2}\\
        &\lesssim(\la\delta)^{\frac{n-2}{2(n+1)}}\Bigl(\sum_{(\omega,\omega')\in C_{\text{bad}}^\gamma\times C_{\text{bad}}^{\gamma'}}\lpn{P_{\gamma} f}{2}\lpn{P_{\gamma'}f}{2}\Bigr)^{1/2}\\
        &\lesssim (\la\delta)^{\frac{n-2}{2(n+1)}+\eta n}\Bigl(\lpn{P_{\gamma} f}{2}\lpn{P_{\gamma'}f}{2}\Bigr)^{1/2}
    \end{align*}

Select $\eta< \frac{1}{4n(n+1)}$ when $n\geq 4$. This proves the result with $\sigma_2 = \frac{1}{4(n+1)}$ when $n\geq 4$. A similar calculation shows it is sufficient to take $\eta<\frac{1}{48}$ when $n=2$ and $n=3$ to prove the result with $\sigma_2=\frac{1}{24}$ and $\sigma_2=\frac{1}{16}$ respectively.
\end{proof}

By selecting $\sigma = \min\{\sigma_1,\sigma_2\}$, Proposition \ref{firstThreeTermsBoundsDecplProp} and Proposition \ref{badbadprop} imply Proposition \ref{arcResult} which bounds the far terms. These choices of $\eta$ also obviously satisfy $\eta<1/2$ which is necessary to apply this version of the cap counting result. For a given $n$ we select $\eta$ that satisfies the discussed conditions. Once this choice is fixed, we choose the $\epsilon$ in the proof of Proposition \ref{firstThreeTermsBoundsDecplProp} to be such that $\la^\epsilon = (\la\delta)^{\frac{\eta}{2(n+1)}}$. This allows us to conclude our bound for the far terms. 

\section{Handling the Diagonal Terms}\label{diagTermsSection}
\subsection{Bounding the diagonal terms assuming microlocal estimates}
To bound the diagonal terms we will use microlocal operators $Q_\nu$ that act on functions on the torus as was done in \cite{BHS1}, specifically Section 6 of that paper. While the operator here will be the same as the one defined in that paper up to a change of variables, the estimates we will use will look slightly different as we have not replaced our projection operators with a smoothed out approximation. In this section we will define the $Q_\nu$ and assume three estimates involving it, then show how this allows us to finish the argument. Section \ref{microProofsSection} is then dedicated to proving these estimates.

Let us define the $Q_\nu$ where $\nu$ is taken from the same set that defined $P_\nu$. Let $\sum \psi(x) =1$ be a smooth partition of unity of $\T{n}$. We demand that the support of each $\psi$ be small enough so that once $\T{n}$ is identified with the fundamental domain of $L^*$ in $\R{n}$, $\psi f$ has a periodic extension via $L^*$ to all of $\R{n}$ for any $f$. This can be done by taking $\supp{\psi}$ to be contained in a ball of sufficiently small radius. This radius only depends on the chosen lattice $L^*$ and dimension $n$. There are $O(1)$ cutoffs $\psi$ and its properties do not depend on the spectral parameter.

For each $\psi$ in this partition define a $\tilde{\psi}$ that is $1$ on the support of $\psi$ and whose support is small enough such that $\tilde{\psi}f$ can be extended for every $f$ in the manner previously described. We shall use the $\psi$ to partition the torus. Technically, there exists a $Q_\nu$ for each $\psi$ but as we will always analyze after applying the triangle inequality we suppress this dependence.

Let $\sum_\nu \beta_\nu(\xi) = 1$ be a partition of unity on $S^{n-1}$ such that $\beta_\nu$ is supported in a $2\theta_0$ cap centered at $\nu\in S^{n-1}$. Also write $\beta_\nu$ for the degree zero homogeneous extension of these bump functions such that they serve as a partition of unity for all $\xi$ away from the origin. Additionally, let $\tilde{\beta}(\xi)$ be a bump function that is supported on $[1/4,4]$ and identically 1 on $[1/2,2]$. Define the kernel

\begin{equation*}
    Q_\nu(x,y) \coloneq \tilde{\psi}(x)\int_{\R{n}}e^{2\pi i\langle x-y,\xi\rangle}\beta_\nu(\xi/|\xi|)\tilde{\beta}(2\pi|\xi/\la|) d\xi.
\end{equation*}
Where $x\in \T{n}$ and $y$ will be integrated over $\T{n}$. It is important for our purposes that $\tilde{\psi}$ is compactly supported in such a way that $Q_\nu f$ is again a well defined function on the torus. In that light, if we identify $\T{n}$ with a tiling of $\R{n}$ we initially define $Q_\nu f$ on the fundamental domain then consider its periodic extension which is well defined because of the properties of $\tilde{\psi}(x)$.

\begin{remark}
    Constructing the $Q_\nu$ in this context obscures the subtlety that went into their construction. This is because of a property of the torus; the geodesic flow is given by straight lines whose direction does not depend on the value of $x$. In the case of general manifolds of non-positive curvature, the symbols of $Q_\nu$ were constructed to be invariant under the geodesic flow and depended both on $x$ and $\xi$. Additionally, $\psi(x)$ selects out a coordinate chart of $\T{n}$, but in the general case the partition of unity also selected out directions $\xi$ as well. Working on the torus simplifies our analysis at the risk of making the definition feel mystical. One should confer \cite{BHS1} Section 4 for an in depth explanation of the construction as well as properties the $Q_\nu$ that are not relevant for this work.
\end{remark}
There are 3 properties of the $Q_\nu$ that we shall use to conclude. 

\begin{proposition}\label{mircoLocalProperties}
 Fix a $\psi$ from the partition of unity and consider its associated mircrolocal cutoffs $Q_\nu$. Suppose $(\nu,\nu')$ are such that $|\nu-\nu'|\geq C\theta_0$ for some fixed $C>1$ that can be taken sufficiently small. Then for every positive integer $N$ we have

\begin{align}
    &\lpnorms{(I-\sum_{\nu}Q_\nu)\psi P_{\la,1}f}{p}{\T{n}} \lesssim_N \la^{-N}\lpnorm{f}{2}{\T{n}}\label{microInsertEqn}\\
    &\lpnorm{Q_{\nu'}\psi P_{\nu} f}{p}{\T{n}}\lesssim_{N} \la^{-N}\lpnorm{f}{2}{\T{n}}\label{microProjectorAngleEqn} \\
    &\sup_{\nu}\lpnorm{Q_{\nu}\psi P_{\la,\delta} f}{ p_c}{\T{n}}\lesssim_\kappa (\la\delta)^{\mu_1(p_c)}\lpnorm{f}{2}{\T{n}}\label{kakeyaNikodymEqn}
\end{align}
\end{proposition}
The first two estimates follow from non-stationary phase and show that the given terms enjoys rapid decay in $\la$, and as such will be treated as unimportant errors. The third estimate is more difficult to prove. It does follow directly from equation (6.18) in \cite{BHS1} after an application of duality and orthogonality. The proof of all three of these estimates will be given in Section \ref{microProofsSection}. For now, let us assume them and see how they allow us to conclude.

\begin{proposition}\label{diagGoalProp}
    Assume the results in Proposition \ref{mircoLocalProperties}. Let $\Upsilon^{\text{diag}}f$ be as above. Then we have

    \begin{equation}
        \lpnorm{(\Upsilon^{\text{diag}}f)^{1/2}}{p_c}{\mathbb{T}^n} = \lpnorm{\Upsilon^{\text{diag}}f}{p_c/2}{\mathbb{T}^n}^{1/2} \lesssim (\la\delta)^{\frac{1}{p_c}}||f||_2.
    \end{equation}
\end{proposition}

\begin{proof}
    Note that for a fixed $\nu$, there are only $O(1)$ values of $ \nu'$ such that $(\nu,\nu')\in \Xi_{\theta_0}$. This fact and the triangle inequality give

    \begin{align*}
        \lpn{\Upsilon^{diag}f}{p_c/2}^{1/2}&= \lpn{\sum_{(\nu,\nu')\in \Xi_{\theta_0}}P_{\nu}fP_{\nu'}f}{p_c/2}^{1/2}  \\
        & \leq \Bigl(\sum_{(\nu,\nu')\in \Xi_{\theta_0}}\lpn{P_{\nu}fP_{\nu'}f}{p_c/2}\Bigr)^{1/2}\\
        &\leq \Bigl(\sum_{(\nu,\nu')\in \Xi_{\theta_0}}\lpn{P_{\nu}f}{p_c}\lpn{P_{\nu'}f}{p_c}\Bigr)^{1/2}\\
        &\leq \Bigl(\sum_{(\nu,\nu')\in \Xi_{\theta_0}}\lpn{P_{\nu}f}{p_c}^2\Bigr)^{1/4}\Bigl(\sum_{(\nu,\nu')\in \Xi_{\theta_0}}\lpn{P_{\nu'}f}{p_c}^2\Bigr)^{1/4}\\
        &\lesssim \Bigl(\sum_{\nu}\lpn{P_{\nu}f}{p_c}^2\Bigr)^{1/2}.
    \end{align*}
We now introduce the partition of unity $\psi$. By the triangle inequality it is sufficient to bound

\begin{equation*}
    \Bigl(\sum_{\nu}\lpn{\psi P_{\nu}f}{p_c}^2\Bigr)^{1/2}
\end{equation*}
By \eqref{microInsertEqn} and the relation $P_{\nu} = P_{\la,1}P_\nu$, we have

\begin{equation*}
    \Bigl(\sum_{\nu}\lpn{\psi P_{\nu}f}{p_c}^2\Bigr)^{1/2}\lesssim_N\Bigl(\sum_{\nu}\lpn{\sum_{\nu'}Q_{\nu'} \psi P_{\nu}f}{p_c}^2\Bigr)^{1/2} + O(\la^{-N'}).
\end{equation*}
Here $O(\la^{-N'})$ indicates a rapidly decreasing error. As the number of $\nu$ is $O(\la^{c_n})$ for some dimensional constant $c_n$, we can select $N$ depending on the dimension large enough in a manner not depending on $\la$ so that $N'$ is a positive power. The error therefore can be ignored. We will suppress such choices of $N$ from this argument and going forward.

By \eqref{microProjectorAngleEqn}, The only $Q_{\nu'}\psi P_\nu$ which do not enjoy rapid decay are the $O(1)$ values of $\nu'$ such that $|\nu-\nu'|<C\theta_0$ for some $C$. This constant can be chosen small as will be elaborated on in the proof of Proposition \ref{mircoLocalProperties}. We therefore get

\begin{equation*}
    \Bigl(\sum_{\nu}\lpn{\sum_{\nu'}Q_{\nu'}\psi P_{\nu}f}{p_c}^2\Bigr)^{1/2}\lesssim \Bigl(\sum_{\nu}\lpn{\sum_{\nu': |\nu'-\nu|<C\theta_0}Q_{\nu'}\psi P_{\nu}f}{p_c}^2\Bigr)^{1/2} + O(\la^{-N''}).
\end{equation*}

The rapidly decreasing error is not the same as before, but it also can be ignored by a sufficiently large choice of $N$. That there $O(1)$ values of $\nu'$ in the inner sum, the fact that $P_\nu = P_{\la,\delta}P_\nu$, and \eqref{kakeyaNikodymEqn} yield

\begin{align*}
    &\Bigl(\sum_{\nu}\lpn{\sum_{\nu': |\nu'-\nu|<C\theta_0}Q_{\nu'}\psi P_{\nu}f}{p_c}^2\Bigr)^{1/2}\lesssim \Bigl(\sum_{\nu}\Bigl(\sum_{\nu': |\nu'-\nu|<C\theta_0}\lpn{Q_{\nu'}\psi P_{\la,\delta} P_{\nu}f}{p_c}\Bigr)^2\Bigr)^{1/2}\\
    &\lesssim \sup_{\nu}\lpn{Q_\nu\psi P_{\la,\delta}}{2\rightarrow p_c}(\sum_{\nu}\lpn{P_{\nu}f}{2}^2\Bigr)^{1/2}\\
    &=(\la\delta)^{\mu_1(p_c)}\lpn{f}{2} = (\la\delta)^{\frac{1}{p_c}}\lpn{f}{2}.
\end{align*}
Where orthogonality was also used on the last line.
\end{proof}
\begin{remark}
That the bound for $\sup_{\nu}\lpn{Q_\nu\psi P_{\la,\delta}}{2\rightarrow p_c}$ lines up with the (sharp) conjectured bounds is what allows us to prove better estimates for $\Upsilon^{\text{far}}$ without worry of a contradiction. Additionally, the estimates for $\sup_{\nu}\lpn{Q_\nu\psi P_{\la,\delta}}{2\rightarrow p}$ can easily be shown to align with the conjecture for every p and every  $\delta > \la^{-1+\kappa}$ for a fixed positive $\kappa$. The consequences of this for estimates involving large $p$ will be explored in future work. If one wishes to take $\kappa=0$, a slight alteration of the construction of $Q_\nu$ gives essentially the same estimate with an additional $O_\epsilon(\la^{\epsilon})$ factor, and so this method does not appear to be able to prove sharp eigenfunction bounds.
\end{remark}

\section{Proving the mircrolocal estimates}\label{microProofsSection}
We have reduced matters to showing the three estimates in Proposition \ref{mircoLocalProperties} holds. We will do this in two steps. The first two will be proven together, with only the first bound written out explicitly, as they both follow from essentially the same argument. After expanding the operators, an application of non-stationary phase yields rapid decay. Proving the last bound is more involved.

\subsection{Proving proposition \ref{mircoLocalProperties}, estimates \eqref{microInsertEqn} and \eqref{microProjectorAngleEqn}} 
Our goal in this section is to prove the following two estimates.
\begin{align*}
    &\lpnorms{(I-\sum_{\nu}Q_\nu)\psi P_{\la,1}f}{p}{\T{n}} \lesssim_N \la^{-N}\lpnorm{f}{2}{\T{n}}\\
    &\lpnorm{Q_{\nu'}\psi P_{\nu} f}{p}{\T{n}}\lesssim_{N} \la^{-N}\lpnorm{f}{2}{\T{n}} \\
\end{align*}

\begin{proof}[Proof of Proposition \ref{mircoLocalProperties}, estimates \eqref{microInsertEqn} and \eqref{microProjectorAngleEqn}]
We will prove the first inequality in detail. The second follows by a similar argument. To begin let us expand $\sum_{\nu}Q_\nu\psi P_{\la,1}f$ as

\begin{equation}
    \tilde{\psi}(x)\sum_{k_L\in A_{\la,1}}\hat{f}(k_L)\int_{\R{n}}\int_{\R{n}}e^{2\pi i\langle x-y,\xi\rangle+2\pi i\langle k_L, y\rangle}\sum_{\nu}\beta_\nu(\xi/|\xi|)\tilde{\beta}(2\pi|\xi/\la|)\psi(y)dy d\xi.
\end{equation}
Ostensibly the integral in $y$ is taken over $\T{n}$. As $\psi(y)$ is of sufficiently small compact support, once we view the integral over $\T{n}$ as an integral over the fundamental domain in $\R{n}$ we can trivially extend the integral to the full space. This will help with alignment with the other term $\psi P_{\la,1}f$. Using that $\tilde{\psi}(x)\psi(x) = \psi(x)$ by the definition of $\tilde{\psi}(x)$, we can expand $\psi P_{\la,1}f$ as the following using a Fourier transform and inverse Fourier transform,

\begin{equation}
        \tilde{\psi}(x)\sum_{k_L\in A_{\la,1}}\hat{f}(k_L)\int_{\R{n}}\int_{\R{n}}e^{2\pi i\langle x-y,\xi\rangle+2\pi i\langle k_L, y\rangle}\psi(y)dy d\xi.
\end{equation}
$\psi(y)$ is compactly supported in $\R{n}$ so the integrals above obviously converge. The difference of these operators is then

\begin{equation}
    \tilde{\psi}(x)\sum_{k_L\in A_{\la,1}}\hat{f}(k_L)\int_{\R{n}}\int_{\R{n}}e^{2\pi i\langle x-y,\xi\rangle+2\pi i\langle k_L, y\rangle}(1-\tilde{\beta}(2\pi|\xi/\la|))\psi(y)dy d\xi.
\end{equation}
Where we have used $\sum_\nu \beta_\nu = 1$ away from the origin to simplify. Let $\phi = 2\pi i\langle x-y,\xi\rangle+2\pi i\langle k_L, y\rangle$. For each integral in the sum we clearly have

\begin{equation*}
    \partial_y\phi = 0 \iff \xi = k_L.
\end{equation*}
Our cutoff $\tilde{\beta}$ has been constructed such that $\xi\neq k_L$ on the $\xi$ support of the integral. The properties of $1-\tilde{\beta}$ force 

\begin{equation*}
    2\pi|\xi|\not\in [\la/2,2\la].
\end{equation*}
As $k_L\in A_{\la,1}$ implies $2\pi |k_L|\in [\la-1,\la+1]$ we clearly have 

\begin{equation*}
    |\xi-k_L|\geq \frac{1}{4}\la,
\end{equation*}
for $\la$ large enough. Repeated integration by parts in $y$ (non-stationary phase) gives the desired bound. To prove the other inequality one repeats this proof except $\xi$ and $k_L$ will be separated because of the restricted angle not because of the restricted radius. The details are omitted.
\end{proof}

\subsection{Proving Proposition \ref{mircoLocalProperties}, estimate \eqref{kakeyaNikodymEqn}}
We need to prove 

\begin{equation}
    \sup_{\nu}\lpnorm{Q_{\nu}\psi P_{\la,\delta} f}{ p_c}{\T{n}}\lesssim_\kappa (\la\delta)^{\mu_1(p_c)}\lpnorm{f}{2}{\T{n}}.
\end{equation}
Using orthogonality and duality, this is an immediate consequence of (6.18) from \cite{BHS1} once the introduced $\psi$ is accounted for. We will include the proof for completeness. To do so we must introduce a smoothed out version of the projection operator $P_{\la,\delta}$. Let $\rho$ be a real-valued Schwartz class function defined  on $\R{}$ satisfying the following two properties.

\begin{equation*}
    \rho(0)=1, \hspace{.1cm} \supp{\hat{\rho}}\subset \{t: |t|\in (c/2,c)\}.
\end{equation*}
Where $c$ is a sufficiently small constant. We then define the following for $f\in L^2(\T{n})$,

\begin{equation}
    \rho_\la f = \sum_{m=0}^\infty\rho(\delta^{-1}(\la-\la_m))E_ mf.
\end{equation}
Here $E_m f$ projections onto the $m$-th eigenspace of $\sqrt{-\Delta_{\T{n}}}$. We decorate the operator with $\T{n}$ to clarify later expressions. By duality and orthogonality it suffices to prove

\begin{equation}\label{KakeyaNikodymBound}
    \sup_{\nu}\lpnorm{Q_{\nu}\psi \rho_\la f}{ p_c}{\T{n}}\lesssim_\kappa (\la\delta)^{\mu_1(p_c)}\lpnorm{f}{2}{\T{n}}.
\end{equation}
The duality and orthogonality arguments we have been relying on are outlined in Theorem 3.2.1 of \cite{S2}. The above is the form proven in \cite{BHS1}. There, the smoothed out projections $\rho_\la$ are considered for the entirety of the argument whereas we keep the projectors discrete until the present step in order to take advantage of decoupling and cap counting arguments.

\begin{proof}[Proof of Proposition \ref{mircoLocalProperties}, estimate \eqref{kakeyaNikodymEqn}]
To aid in the exposition we shall change variables so that

\begin{equation*}
    Q_\nu(x,y) =\tilde{\psi}(x)\la^n\int_{\R{n}}e^{2\pi i\la\langle x-y,\xi\rangle}\beta_\nu(\xi/|\xi|)\tilde{\beta}(2\pi|\xi|) d\xi.
\end{equation*}
We will import the following two estimates from \cite{BHS1}.

\begin{align*}
    &\lpn{Q_\nu}{L^p(\T{n})\rightarrow L^p(\T{n})}\lesssim 1 \text{ for } 2\leq p \leq \infty,\\
    &\lpn{Q_\nu^*}{L^p(\T{n})\rightarrow L^p(\T{n})}\lesssim 1 \text{ for } 1\leq p \leq 2.
\end{align*}
By $TT^*$ and an application of the above bound for the adjoint, if $\Psi = \rho^2$ it is sufficient to prove

\begin{equation}\label{TTstarbound}
    \lpn{Q_\nu\psi \Psi}{p_c'\rightarrow p_c}\lesssim (\la\delta)^{2\mu_1(p_c)}.
\end{equation}

Let $\beta$ be a smooth compactly supported bump function supported on $(1/2,2)$ such that for $\beta_k \coloneq \beta(t/2^k)$ we have $\sum_{k=-\infty}^\infty\beta_k(t) = 1$ for $t>0$. Also define $\beta_0\coloneq 1- \sum_{k=0}^\infty\beta_k$.

By Fourier's inversion formula, we have the decomposition

\begin{equation}
    \Psi = L_\la + G_\la,
\end{equation}
where 

\begin{equation}\label{localOperator}
    L_\la = \delta\int e^{2\pi i\la t}e^{-2\pi it\sqrt{-\Delta_{\T{n}}}}\beta_0(2\pi|t|)\hat{\Psi}(\delta 2\pi t)dt,
\end{equation}
and 
\begin{equation}\label{globalOperator}
    G_\la = \delta\int e^{2\pi i\la t}e^{-2\pi it\sqrt{-\Delta_{\T{n}}}}(1-\beta_0)(2\pi|t|)\hat{\Psi}(\delta 2\pi  t)dt.
\end{equation}
We interpret the operator $e^{-2\pi it\sqrt{-\Delta_{\T{n}}}}$ by the spectral theorem, that is

\begin{equation*}
    e^{-2\pi it\sqrt{-\Delta_{\T{n}}}}f = \sum_{m=0}^\infty e^{-2\pi it\la_j}E_m f.
\end{equation*}
By Euler's formula we have that 

\begin{equation*}
    e^{-2\pi it\sqrt{-\Delta_{\T{n}}}} = 2\cos(2\pi t\sqrt{-\Delta_{\T{n}}}) + e^{2\pi it\sqrt{-\Delta_{\T{n}}}}.
\end{equation*}
If we replace $ e^{-2\pi it\sqrt{-\Delta_{\T{n}}}}$ by $e^{2\pi it\sqrt{-\Delta_{\T{n}}}}$, the resulting operator has a kernel that enjoys rapid decay in $\la$ (bounded by $\lesssim_N\la^{-N}$ for any integer $N$). This is because the signs of the two exponentials agree and so there are no stationary points. As such, if 

\begin{equation}\label{globalOperator2}
    \tilde{G}_\la = 2\delta\int e^{2\pi i\la t}\cos(2\pi t\sqrt{-\Delta_{\T{n}}})(1-\beta_0)(2\pi |t|)\hat{\Psi}(\delta 2\pi  t)dt,
\end{equation}
it suffices to bound $Q_\nu \psi  L_\la$ and $Q_\nu\psi \tilde{G}_
\la$ by the bound in \eqref{TTstarbound}. 

To bound $Q_\nu \psi L_\la$ we do not need the microlocal cutoffs, and so we will just bound $L_\la$ and use that $Q_\nu \psi$ is $L^{p_c}\rightarrow L^{p_c}$ bounded. After using the spectral theorem to expand we have

\begin{align*}
    & L_\la = \delta\sum_{m=0}^\infty \Bigr(\int e^{2\pi it(\la-\la_m)}\beta_0(2\pi |t|)\hat{\Psi}(\delta 2\pi t) dt\Bigl) E_m\\& = \delta \sum_m s(\la;\la_m)E_m.
\end{align*}
Where $s(\la;\la_m)$ is the inverse Fourier transform of $\beta_0(2\pi t)\hat{\Psi}(\delta2\pi  t)$ with respect to $t$. Because of the fixed compact support of $\beta_0(2\pi t)$ the product $\beta_0(2\pi t)\hat{\Psi}(\delta2\pi  t)$ has compact support independent of $\la$. Its inverse Fourier transform is therefore a Schwartz class function, and the multiplier $s(\la;\la_m)$ enjoys the bound

\begin{equation*}
    |s(\la;\la_m)|\lesssim_N (1+|\la-\la_m|)^{-N},
\end{equation*}
for any choice of $N$ which leads to $\sum_m s(\la;\la_m)E_m$ obeying the universal estimates. The universal estimate at $p_c$ is the bound for spectral projectors of with 1 in the following expression

\begin{equation}
    \lpn{P_{\la,1}}{p_c'\rightarrow p_c}\lesssim\la^{2\mu_1(p_c)}.
\end{equation}

See Theorem 2.1 in \cite{S3}. To see that the universal estimates apply we break up the multiplier based on $|\la-\la_m|$. In particular,

\begin{equation}
     L_\la = \delta \sum_m s(\la;\la_m)E_m = \delta \sum_{m: |\la-\la_m|<1}s(\la;\la_m)E_m + \delta \sum_{j=1}^\infty\sum_{m: |\la-\la_m|\sim 2^j}s(\la;\la_m)E_m.
\end{equation}
As the multiplier $s(\la;\la_m)$ is bounded the first terms obeys the universal estimates. For the second terms, the set $|\la-\la_m|\sim 2^s$ can be broken into $2^s$ intervals of length $\sim 1$ that obey the universal estimates. The rapid decay of $s(\la;\la_m)$ ensures the multiplier is bounded by $2^{-jN}$ for any $N$ of our choosing. Therefore, by the triangle inequality.

\begin{equation}
    \lpn{L_\la}{p_c'\rightarrow p_c}\lesssim_N \delta\la^{2\mu_1(p_c)} +  \delta\la^{2\mu_1(p_c)}\sum_j 2^{-Nj}2^j \lesssim \delta\la^{2\mu_1(p_c)}.
\end{equation}
For $N$ chosen sufficiently large in a manner that does not depend on $\la$. We then have by $TT^*$ that 

\begin{equation}
    \lpn{L_\la}{2\rightarrow p_c}\lesssim \delta^{1/2}\la^{\mu_1(p_c)}.
\end{equation}
To bound $G_\la$ we will continue to use our time cutoff $\beta_k$. In particular, define 

\begin{equation}
    \tilde{G}_{\la,k} = 2\delta\int e^{2\pi i\la t}\cos(2\pi t\sqrt{-\Delta_{\T{n}}})\beta_k(2\pi |t|)\hat{\Psi}(\delta2\pi  t)dt.
\end{equation}
By Huygens' principle we can restrict to $2^k \lesssim \delta^{-1}$. By summing in $k$, we have reduced to showing

\begin{equation}
    \lpn{ Q_\nu\psi \tilde{G}_{\la,k}}{p_c'\rightarrow p_c}\lesssim \delta\la^{2\mu_1(p_c)}2^{\frac{2k}{n+1}}.
\end{equation}
We will prove this by interpolation. By basic properties of the Fourier transform we have 

\begin{equation}
    \lpn{Q_\nu\psi \tilde{G}_{\la,k}}{2\rightarrow 2}\lesssim \delta 2^k.
\end{equation}
We will interpolate this with a $L^1\rightarrow L^\infty$ norm which will be more difficult to prove. Our goal is
\begin{equation}\label{1toinftyBoundGoal}
    \lpn{Q_\nu \psi \tilde{G}_{\la,k}}{1\rightarrow \infty}\lesssim \delta\la^{\frac{n-1}{2}}2^{k(1-\frac{n-1}{2})}.
\end{equation}
This is the only estimate that requires the $Q_\nu$ operators. To evaluate the kernel of $\cos(2\pi t\sqrt{-\Delta_{\T{n}}})$ we will lift to the universal cover which is a tiling of $\R{n}$ by the fundamental domain of $L^*$. That is,

\begin{equation}
    \cos(2\pi t\sqrt{-\Delta_{\T{n}}})(x,y) = \sum_{j\in\Z{n}} \cos(2\pi t\sqrt{-\Delta_{\R{n}}})(x-(y+j_{L^*})).
\end{equation}
With $j_{L^*} = j_1l_1^*+ ... + j_nl_n^*$ where $l_i^*$ are the vectors that define the fundamental domain of $\T{n}$. As $j$ runs through $\Z{n}$, one can think of $j_{L^*}$ as iterating through the domains in the universal cover of $\T{n}$. Let $K_{\la,k}(x,y)$ be the kernel of $\tilde{G}_{\la,k}$. For clarity we will first analyze without the microlocal cutoffs. Applying the above formula gives

\begin{equation}\label{kernelDef}
    K_{\la,k}(x,y) = \sum_{j\in \Z{n}}\tilde{K}_{\la,k}(x,(y+j_{L^*})),
\end{equation}
where

\begin{align*}
    &\tilde{K}_{\la,k}(x,y) = 2\delta\int e^{2\pi i\la t}\cos(2\pi t\sqrt{-\Delta_{\R{n}}})(x,y)\beta_k(2\pi |t|)\hat{\Psi}(\delta2\pi  t)dt\\
    &=2\delta\int \int e^{2\pi i\la t}e^{2\pi i\langle x-y,\xi\rangle }\cos(2\pi|\xi|)\beta_k(2\pi |t|)\hat{\Psi}(\delta2\pi  t)dtd\xi.
\end{align*}
Here, $x\in \T{n}$ is the variable we are integrating over in the norm whereas $y\in \T{n}$ is the kernel variable. This kernel vanishes if $|x-y|\geq 2^{k+1}$. By non-stationary phase the kernel is rapidly decreasing in $\la$ when $|x-y|\leq 2^{k-1}$. As such, we will assume $|x-y|\sim 2^k$. This selects out the $k$ in the sum of \eqref{kernelDef} that significantly contribute.

We will now compute a crude estimate that will motivate our particular definition of the $Q_\nu$. By stationary phase (see for example Lemma 5.13 in \cite{S1}) we have

\begin{equation}
    \tilde{K}_{\la,k}(x,y) = \sum_{\pm}\delta \la^{\frac{n-1}{2}}2^{-k\frac{n-1}{2}}e^{\pm 2\pi i \la|x-y|}a_{\pm}(\la,|x-y|) +  R_\la.
\end{equation}
Where $R_\la$ is an error term that enjoys rapid decay. To write this expression we need the reduction $|x-y|\sim 2^k$. By our previous reductions, the only contributing domains correspond to $j_{L^*}$ such that $|x-(y+j_{L^*})|\sim 2^k$ of which there are $O(2^{kn})$ such vectors. This gives 

\begin{equation}
    |K_{\la,k}(x,y)|\lesssim 2^{kn}\delta \la^{\frac{n-1}{2}}2^{-k\frac{n-1}{2}} = \delta \la^{\frac{n-1}{2}}2^{k\frac{n+1}{2}}.
\end{equation}
As its kernel obeys this bound we also have that $\lpn{\tilde{G}_{\la,k}}{1\rightarrow \infty}\lesssim \delta \la^{\frac{n-1}{2}}2^{k\frac{n+1}{2}}$. This bound is too weak to allow us to conclude when compared against \eqref{1toinftyBoundGoal}. To get an improved bound we need to limit the number of contributing domains. This is what the $Q_\nu$ allow us to do. Essentially, only domains contained in the angular support of the symbol of $Q_\nu$ will significantly contribute. We will count domains by counting lattice points as there is a one-to-one correspondence. As the support of this symbol is contained in a $\sim \delta$ aperture cone centered at the origin, instead of considering all lattice points in a ball of radius $2^k$ centered at $x-y$ which is within the fundamental domain, we instead consider the set of points in a ball of radius $2^k$ centered at $x-y$ that also make an angle $\lesssim \delta$ with $\nu$ as all other points do not significantly contribute. 

The number of such points is $\lesssim \delta^{n-1}2^{kn}\lesssim \delta^{-1}$ which is a significant improvement over out initial crude bound and is strong enough to allow us to conclude \eqref{1toinftyBoundGoal}. We now turn to the details.

First we write the kernel of $Q_\nu \psi\tilde{G}_{\la,k}$ as
\begin{align*}
    Q_\nu \psi\tilde{G}_{\la,k}(x,y) &= 2\delta\sum_{j\in \Z{n}}\int e^{2\pi i\la t}(Q_\nu \psi \cos(2\pi t\sqrt{-\Delta_{\R{n}}})(x,y+j_{L^*}))\beta_k(2\pi |t|)\hat{\Psi}(\delta2\pi  t)dt\\
    &=\sum_{j\in \Z{n}}K_{\la,k,\nu}(x,y+j_{L^*}).
\end{align*}
As a slight technical note, $Q_\nu$ is the pullback of $Q_\nu$ to the fundamental domain of $L^*$. This operator has the same Schwartz kernel as our original definition of $Q_\nu$ and so we suppress the difference.

Fix $x,y\in \T{n}$ (viewed as the fundamental domain of $L^*$) and $k$. Define

\begin{align}
    &J_{\text{main}}\coloneq\Bigl\{j\in \Z{n}: \Bigl|\frac{x-(y+j_{L^*})}{|x-(y+j_{L^*})|}-\nu\Bigr|\leq C\delta, |x-(y+j_{L^*})|\sim 2^k\Bigr\}\\
    &J_{\text{error}}\coloneq\Bigl\{j\in \Z{n}: \Bigl|\frac{x-(y+j_{L^*})}{|x-(y+j_{L^*})|}-\nu\Bigr|\geq C\delta, |x-(y+j_{L^*})|\sim 2^k\Bigr\}.
\end{align}
Here, the constant $C$ is selected to align with the $\xi$ support of the symbol of $Q_\nu$ and can be selected small enough. 

The main term is handled by a counting argument. We have $\# J_{\text{main}} = O(2^{kn}\delta^{n-1})$ as any $j$ in this set must be such that $j_{L^*}$ is contained in the intersection of a ball of radius $2^k$ centered at $x-y$ and a cone of angular aperture $C\delta$ centered at $x-y$. As $Q_\nu\psi$ is $L^\infty\rightarrow L^\infty$ bounded, we have

\begin{align}
    &\sum_{j\in J_{\text{main}}}|Q_\nu\psi K_{\la,k}(x,y)|\lesssim \sum_{j\in J_{\text{main}}}|K_{\la,k}(x,y)|\\ 
    &\lesssim \sum_{j\in J_{\text{main}}}\delta \la^{\frac{n-1}{2}}2^{-k\frac{n-1}{2}}\\
    &\lesssim \delta^n\la^{\frac{n-1}{2}}2^{k(n-\frac{n-1}{2})}.
\end{align}
Which, as $2^j\lesssim \delta^{-1}$ aligns with the bound we want. We now only have to show that the terms corresponding to $J_{\text{error}}$ do not contribute significantly.

We handle these terms by a stationary phase argument. To do this, fix $j_{L^*}$ and expand $K_{\la,k,\nu}(x,y+j_{L^*})$ up to a rapidly decreasing error as

\begin{equation}
    C_{\la,\delta,k}\tilde{\psi}(x)\int_{\T{n}} \int_{\R{n }} e^{2\pi i\la\langle x-z, \xi\rangle\pm 2\pi i\la |z-(y+j_{L^*})|}a_{\pm}(\la,|z-(y+j_{L^*})|)\beta_\nu(\xi/|\xi|)\beta(2\pi|\xi|)\psi(z)d\xi dz.
\end{equation}
As in part 1 of the proof we may extend the integral in $z$ from $\T{n}$ to $\R{n}$ by the support properties of $\psi(z)$. We also note that $C_{\la,\delta, k} = O(\delta \la^{\frac{3n-1}{2}}2^{-k\frac{n-1}{2}})$.

Because we are in a neighborhood of the $(x,z)$ diagonal, our previous reduction to only consider $(x,y+j_{L^*})$ such that $|x-(y+j_{L^*})|\sim 2^k$ implies we may reduce to considering $(z,y+j_{L^*})$ such that $|z-(y+j_{L^*})|\sim 2^k$.

Our condition on $J_{\text{error}}$ is such that 

\begin{equation}
    \Bigl|\frac{x-(y+j_{L^*})}{|x-(y+j_{L^*})|}-\nu\Bigr|\geq C\delta.
\end{equation}
Because of our cutoff function $\tilde{\psi}(x)$, the variable $x$ is restricted to a neighborhood of $z$. Additionally, the angular support of $\xi$ is constrained to a $\sim \delta$ neighborhood of $\nu$. Simple geometry shows that that the above conditions imply

\begin{equation}
    \Bigl|\frac{z-(y+j_{L^*})}{|z-(y+j_{L^*})|}-\xi\Bigr|\gtrsim \delta.
\end{equation}
Using this we can integrate by parts in $z$. After each iteration we gain by a factor of $(\la\delta)^{-1}$. The factor of $\delta^{-1}$ comes form the support properties of $\beta_\nu$. As we have assumed $\la\delta = \la^{\kappa}$ for a fixed $\kappa>0$ each application of integration by parts allows us to gain by $\la^{-k}$ resulting in a factor that is $O_N(\la^{-\kappa N})$ for any positive integer $N$. There are fixed polynomial power of $\la$ number of $j\in J_{\text{error}}$ and $C_{\la,\delta,k}$ is also a fixed polynomial power of $\la$ and so by choosing $N$ large enough we have that 

\begin{equation}
    \sum_{j\in J_{\text{error}}}|Q_\nu K_{\la,j}(x,y)|\lesssim_N \la^{-M},
\end{equation}
for any integer $M$. So $J_{\text{main}}$ contains all values that significantly contribute and this allows us to conclude. 
\end{proof}

\newcommand{\Addresses}{{% additional braces for segregating \footnotesize
  \bigskip
  \footnotesize

  (D. P.) \textsc{Department of Mathematics, Johns Hopkins University,
    Baltimore, Maryland, 21218}\par\nopagebreak
  \textit{E-mail address: }\texttt{dpezzi1@jhu.edu}

  \medskip
}}

\Addresses


\begin{thebibliography}{6}


\bibitem{BHS1} M. Blair, X. Huang, and C. Sogge, \textit{Improved spectral projection estimates}. preprint, arXiv:2211.17266, to appear in Journal of the European Mathematical Society.

\bibitem{BS1} M. Blair and C. D. Sogge, \textit{Logarithmic improvements in $L^p$ bounds for eigenfunctions at the
critical exponent in the presence of nonpositive curvature}. Invent. Math., 217(2):703–748, 2019.

\bibitem{B1} J. Bourgain, \textit{Fourier transform restriction phenomena for certain lattice subsets and applications to nonlinear evolution equations. Part I: Schr\"odinger equations}, Geom. Funct. Anal., 3 (2):107-156, (1993).

\bibitem{BD1}J. Bourgain and C. Demeter, \textit{The proof of the $l^2$ decoupling conjecture}. Annals of mathematics, pages 351–389, 2015.

\bibitem{D1} C. Demeter, Fourier Restriction, Decoupling and Applications, Cambridge Stud. Adv. Math., 184 Cambridge University Press, Cambridge, 2020

\bibitem{DG1} C. Demeter and P. Germain, \textit{$L^2$ to $L^p$ bounds for spectral projectors on the Euclidean two-dimensional torus}. arXiv:2306.14286

\bibitem{GM1} P. Germain and S. L. Rydin Myerson, \textit{Bounds for spectral projectors on tori}. Forum Math. Sigma, 10:Paper No. e24, 20, 2022.

\bibitem{H1} J. Hickman, \textit{Uniform $L^p$ resolvent estimates on the torus.} Mathematics Research Reports, 1:31–45, 2020.

\bibitem{HS1}X. Huang and C. Sogge, \textit{Curvature and sharp growth rates of log-quasimodes on compact manifolds}. preprint, arXiv:2404.13734

\bibitem{S3} C. Sogge, \textit{Concerning the $L^p$ norm of spectral clusters for second-order elliptic
operators on compact manifolds}, J. Funct. Anal. 77 (1988), no. 1, 123–138.

\bibitem{S1} C. Sogge, Fourier integrals in classical analysis, volume 105 of Cambridge Tracts in Mathematics. Cambridge University Press, Cambridge, 1993

\bibitem{S2} C. Sogge, Hangzhou lectures on eigenfunctions of the Laplacian, volume 188 of Annals of Mathematics Studies. Princeton University Press, Princeton, NJ, 2014.

\bibitem{TV1} T. Tao and V. Vu, Additive Combinatorics. Cambridge Stud. Adv. Math., 105 Cambridge University Press, Cambridge, 2006
\end{thebibliography}
\end{document}